\documentclass{article}
\usepackage{amssymb}
\usepackage[margin=2.5cm]{geometry}
\usepackage{amsmath}
\usepackage{amsthm}
\usepackage{mathrsfs}
\usepackage{verbatim}
\usepackage{enumerate}
\usepackage{subcaption}
\usepackage{tikz}
\usepackage{listings}
\usepackage{float}
\usepackage{marvosym}

\usepackage{fancyhdr}
\newtheorem{defn}{Definition}[section]
\newtheorem{lemma}[defn]{Lemma}

\newtheorem{prop}[defn]{Proposition}

\newtheorem{theorem}[defn]{Theorem}

\makeatletter
\DeclareRobustCommand{\pder}[1]{%
  \@ifnextchar\bgroup{\@pder{#1}}{\@pder{}{#1}}}
\newcommand{\@pder}[2]{\frac{\partial#1}{\partial#2}}
\makeatother

\begin{document}

\title{Maximal displacement of simple random walk bridge on Galton-Watson trees}

\author{Josh Rosenberg\thanks{This work was supported by a Zuckerman STEM Postdoctoral Fellowship, as well as by ISF grant 1207/15, and ERC starting grant 676970 RANDGEOM.}}
\date{}
\maketitle

\begin{abstract}
We analyze simple random walk on a supercritical Galton-Watson tree, where the walk is conditioned to return to the root at time $2n$.  Specifically, we establish the asymptotic order (up to a constant factor) as $n\to\infty$, of the maximal displacement from the root.  Our results, which are shown to hold for almost surely every surviving tree $T$ (subject to some mild moment conditions), are broken up into two cases.  When the offspring distribution takes a value less than or equal to $1$ with positive probability, the maximal displacement of the bridge is shown to be on the order of $n^{1/3}$.  Conversely, when the offspring distribution has minimum possible value equal to at least $2$ (and is non-constant), the maximal displacement is shown to be of order less than $n$, but greater than $n^{\gamma}$ (for any $\gamma<1$).  Each of these cases is in contrast to the case of a regular tree, on which the bridge is known to be diffusive.  To obtain our results, we show how the walk tends to gravitate towards large clusters of vertices of minimal degree, where it then proceeds to spend most of its time.  The size and frequency of such clusters is generally dependent on the minimum possible value attainable by the offspring distribution, and it is this fact which largely accounts for the existence of the two regimes.
\end{abstract}

\section{Introduction.}

The simple random walk bridge of length $2n$ on a rooted graph refers to a simple random walk (begun at the root) that is conditioned to return to the root at time $2n$.  The study of such processes has, up to this point, mostly been restricted to graphs with certain nice symmetry properties (i.e. transitivity).  Specifically, several papers have examined this process on Cayley graphs, focusing in particular on the relationship between the underlying group structure, and certain properties of the bridge, such as its range (i.e. the number of distinct vertices visited) and the distance from the root at time $n$ (see \cite{BIK, AE}).  By contrast, studies of the bridge in a random environment have generally involved random networks, or other related models, for which the graph itself is fixed.  One such case was addressed by Gantert and Peterson in \cite{GP}, where they analyzed the behavior of the bridge on the graph $\mathbb{Z}$, equipped with i.i.d.~discrete random transition probabilities assigned to its vertices, and established the existence of sub-diffusive, diffusive, and super-diffusive regimes.

In this paper we analyze the maximal displacement of the bridge on Galton-Watson trees and establish that, much like with the model in \cite{GP}, the bridge process can be almost surely diffusive (the settled case of the regular tree), sub-diffusive, or ``nearly" ballistic, depending on the properties of the offspring distribution that generates the random tree.  In order to state these results, we first need to introduce some notation that we will use throughout the paper: Allow $Z$ to refer to an offspring distribution, while letting ${\sf GW}$ denote the measure on Galton-Watson trees associated with $Z$.  When referring to a fixed tree we'll denote it as $T$, whereas a random tree selected according to ${\sf GW}$ will be expressed as ${\bf T}$.  The root of any tree $T$ will be denoted as ${\bf 0}$, and for any vertex $v\in T$, $|v|$ will refer to the height of $v$.  Simple random walk starting at ${\bf 0}$ in $T$ will be denoted as $\{X_n\}$.  Finally, we refer to the measure that $\{X_n\}$ induces on the path space of $T$ as ${\sf SRW}_T$.

Having equipped ourselves with the above definitions, we can now state our main results in the form of the following theorem.

\begin{theorem}\label{theorem:mra3p} \emph{(main results)}.

\medskip
\noindent
(i) If ${\bf P}(Z\leq 1)>0$, and there exists $\delta>0$ such that ${\bf E}[Z^{1+\delta}]<\infty$, then $$\underset{A\to\infty}{\emph{lim}}\Bigg[\underset{n\to\infty}{\emph{liminf}}\ {\sf SRW}_{\bf{T}}\bigg(\frac{1}{A}n^{1/3}\leq\underset{j\leq 2n}{\emph{max}}|X_j|\leq An^{1/3}\ \Big|\ X_{2n}={\bf 0}\bigg)\Bigg]=1\ \ {\sf GW}-\emph{a.s.}$$conditioned on non-extinction of ${\bf T}$.

\medskip
\noindent
(ii) If ${\bf E}[Z]<\infty$ and ${\bf P}(Z\geq 2)=1$, then for every $\gamma<1$
$$\underset{n\to\infty}{\emph{lim}}{\sf SRW}_{\bf{T}}\bigg(\underset{j\leq 2n}{\emph{max}}|X_j|\geq n^{\gamma}\ \Big|\ X_{2n}={\bf 0}\bigg)=1\ \ {\sf GW}-\emph{a.s.},$$and for every $\beta<1$ $$\underset{n\to\infty}{\emph{lim}}{\sf SRW}_{\bf{T}}\bigg(\underset{j\leq 2n}{\emph{max}}|X_j|\geq\frac{n}{(\emph{log}\ n)^{\beta}}\ \Big|\ X_{2n}={\bf 0}\bigg)=0\ \ {\sf GW}-\emph{a.s.}$$
\end{theorem}

\bigskip
\noindent
The proof of part $(i)$ of the theorem is divided into two main parts.  The first part, which is presented as Theorem \ref{theorem:diffbndp1p} at the beginning of the next section, involves addressing the case where ${\bf P}(Z=0)=0$.  The strategy that we use here entails obtaining an almost sure asymptotic lower bound on ${\sf SRW}_{\bf T}(X_{2n}={\bf 0})$ that takes the form of a stretched exponential with exponent of order $n^{1/3}$.  This is achieved by estimating the probability of specific types of events that involve the walk spending nearly all of its time in long stretches of degree one vertices that are distance of order $n^{1/3}$ from the root (where the `degree' of a vertex $v$, sometimes denoted below as $\text{deg}(v)$, refers to its number of children).  From here we use some almost sure properties of Galton-Watson trees, along with a large deviations type argument, to obtain aymptotic upper bounds on the probability that a random walk, begun at any vertex on level $A n^{1/3}$ of ${\bf T}$, ever returns to the root (see Proposition \ref{prop:lbonpolt}).  Comparing these bounds as $A\to\infty$ to the lower bounds achieved for ${\sf SRW}_{\bf T}(X_{2n}={\bf 0})$ then allows us to complete the proof of the upper bound in $(i)$.  To address the lower bound (still looking at the ${\bf P}(Z=0)=0$ case) we define a coupling between simple random walk on a tree $T$, and simple random walk on $\mathbb{Z}$, in order to get an upper bound on the probability that the maximal displacement of the walk up to time $2n$ is less than $\frac{1}{A} n^{1/3}$.  We then compare this bound as $A\to \infty$ to our lower bound on ${\sf SRW}_{\bf T}(X_{2n}={\bf 0})$ in order to complete the proof.

In establishing $(i)$ for the ${\bf P}(Z=0)>0$ case, which is Theorem \ref{brbnd0c}, we find that the methods used in Section 2 to prove the upper bound can be adapted to this new case without a lot of additional work.  Conversely, the task of proving the lower bound presents a number of new challenges.  In fact, showing that the existence of finite subtrees does not tend to reduce the maximal displacement of the bridge by more than a constant factor, turns out to be the most difficult part of the paper.  The approach we use to accomplish this involves treating any surviving tree $T$ as an infinite tree with no leaves, to which we attach finite subtrees to the vertices.  Simple random walk on $T$ can then be thought of as a simple random walk on the infinite part of $T$, that makes excursions into the finite subtrees.  Since it is already established in Section 2 that the maximal displacement of simple random walk on an infinite tree with no leaves, up to time $2n$, is at least on the order of $n^{1/3}$, the main task in completing the proof of $(i)$ is to show that the time the simple random walk on ${\bf T}$ spends on its excursions does not excede the time it spends in the infinite part of ${\bf T}$ by more than a constant factor.  This is done by first using an inductive argument to obtain an upper bound on the exponential moments of certain hitting times for a random walk on a finite tree $T^f$ (see Lemma \ref{lemma:SRWrtft}), then establishing in Lemma \ref{lemma:rficem} the existence of exponential moments for a random variable $r$ related to the size of the random finite tree ${\bf T}^f$ (meaning ${\bf T}$ conditioned to go extinct), and then finally using these results in conjunction with an argument that involves exploiting properties of the annealed distribution ${\sf GW}\times{\sf SRW}_{{\bf T}}$, in order to achieve several asymptotic bounds  on event probabilities (see Proposition \ref{prop:3ineqforlb}), which are then used to obtain the desired result for the quenched case.

In addressing $(ii)$, which appears as Theorem \ref{theorem:diffbound} in Section 4, the proof of the lower bound largely consists of adapting a method that Gantert and Peterson used in \cite{GP} when analyzing the bridge for random walk on $\mathbb{Z}$ (equipped with random transition probabilities).  Our adapted version of this method, when applied to Galton-Watson trees, involves defining a new measure ${\sf BRW}_T$ associated with a particular biased random walk on $T$ denoted as $\{X_j\}^{\beta}$, that is constructed in such a way that ${\sf BRW}_T$ satisfies a certain set of inequalities involving ${\sf SRW}_T$ (see Lemma \ref{lemma:compSRWm}).  We are able to then use an approach that entails estimating the sizes and frequencies of $m$-ary subtrees in a random tree ${\bf T}$, as well as the duration of time the biased random walk tends to spend inside of these subtrees, in order to obtain almost sure asymptotic estimates for ${\sf BRW}_{{\bf T}}\Big(\underset{j\leq 2n}{\text{max}}|X_j|\leq n^{\gamma},\ X_{2n}={\bf 0}\Big)$.  Combining this with the inequalities in Lemma \ref{lemma:compSRWm}, we can approximate the value of ${\sf SRW}_{{\bf T}}\Big(\underset{j\leq 2n}{\text{max}}|X_j|\leq n^{\gamma},\ X_{2n}={\bf 0}\Big)$, allowing us to establish the lower bound in $(ii)$.  We then complete the proof by putting these estimates together with direct estimates of ${\sf SRW}_{{\bf T}}\Big(\underset{j\leq 2n}{\text{max}}|X_j|\leq \frac{n}{(\text{log}\ n)^{\beta}},\ X_{2n}={\bf 0}\Big)$, to obtain the corresponding upper bound.

\section{Case 1(a): ${\bf P}(Z=1)>0$, ${\bf P}(Z=0)=0$}

\begin{theorem}\label{theorem:diffbndp1p}
If the offspring distribution $Z$ satisfies ${\bf P}(Z=0)=0$, $0<{\bf P}(Z=1)<1$, and there exists $\delta>0$ such that ${\bf E}[Z^{1+\delta}]<\infty$, then it follows that\begin{equation}\label{limliminftsin}\underset{A\to\infty}{\emph{lim}}\Bigg[\underset{n\to\infty}{\emph{liminf}}\ {\sf SRW}_{\bf{T}}\bigg(\frac{1}{A}n^{1/3}\leq\underset{j\leq 2n}{\emph{max}}|X_j|\leq An^{1/3}\ \Big|\ X_{2n}={\bf 0}\bigg)\Bigg]=1\ \ {\sf GW}-\emph{a.s.}\end{equation}
\end{theorem}

\bigskip
\noindent
The first step in proving \ref{theorem:diffbndp1p} will consist of achieving a lower bound on the value of ${\sf SRW}_T(X_{2n}=\bf{0})$.  To do this, we'll analyze cases in which the random walk enters a long stretch of the tree in which vertices only have one child (such stretches will be referred to as ``traps"), stays there for almost the entire duration of the walk, and then returns to the root at time $2n$.  Our initial task will therefore be to come up with asymptotic estimates for the lengths of these traps.

Before presenting the first lemma, we provide the following necessary definitions: First, for any tree $T$ and any vertex $v\in T$, let $d_T(v)$ represent the maximum length of a series of degree-1 vertices starting at $v$ and going away from the root, i.e. $d_T(v)$ is the minimum number of steps, starting at $v$, that one must take away from the root before hitting a vertex with at least two children.  Now for every $n\geq 0$, we define $D_{n,k}(T):=\text{max}\{d_T(v):|v|=n,\ \text{deg}(u)\leq k\ \forall\ u<v\}$ (where the expression $u<v$ indicates that the vertex $u$ is an ancestor of the vertex $v$).

\medskip
\begin{lemma}\label{lemma:maxdepthline}
For any tree $T$, let $T^{(k)}$ represent the tree we obtain if we remove all descendants (starting with children) of vertices $v$ for which $\text{deg}(v)>k$, and then let $\mathscr{A}_k$ represent the set of all $T$ (without leaves) for which $T^{(k)}$ survives to infinity.  Additionally, let $\sigma:=\frac{\emph{log}(1/\mu)}{\emph{log}\ \rho}$, where $\mu$ and $\rho$ represent ${\bf E}[Z]$ and ${\bf P}(Z=1)$ respectively.  Then for every $\epsilon>0$, there exists a value $N_{\epsilon}$ such that for $k\geq N_{\epsilon}$ we have\begin{equation}\label{ratiodnkpol}\underset{n\to\infty}{\emph{liminf}}\frac{D_{n,k}({\bf T})}{\sigma n}\geq 1-\epsilon\ \ {\sf GW}-\emph{a.s}.\end{equation} conditioned on $\mathscr{A}_k$.
\end{lemma}

\begin{proof}
We start by defining $\tilde{T}^{(k)}$ to be the tree we obtain by keeping only those vertices in $T^{(k)}$ that are the roots of surviving subtrees (hence, if $T^{(k)}$ does not survive then $\tilde{T}^{(k)}$ is empty).  Next, we let $D^*_{n,k}=\text{max}\{d_T(v):|v|=n,\ v\in \tilde{T}^{(k)}\}$, and note that $D^*_{n,k}\leq D_{n,k}$.  Now defining $Z^{(k)}_n(T)$ to be the size of the $n$th generation of $\tilde{T}^{(k)}$, letting ${\sf GW}^{(k)}:={\sf GW}(\cdot|\mathscr{A}_k)$ for any $k$ satisfying ${\bf E}[Z\cdot 1_{Z\leq k}]>1$, and noting that ${\sf GW}^{(k)}(\text{deg}({\bf 0})=1)=\rho$, we find that for any $\ell>0$, \begin{equation}\label{bndingpl1c}{\sf GW}^{(k)}\big(D^*_{n,k}<\ell\ |\ Z^{(k)}_n\big)=\Big(1-\rho^{\ell}\Big)^{Z^{(k)}_n}.\end{equation}Next we observe that if we have any $r,c$ such that $r>1$ and $0<c<\frac{\text{log}(1/r)}{\text{log}\ \rho}$, then it follows from \eqref{bndingpl1c} that ${\sf GW}^{(k)}(D_{n,k}^*<cn|Z_n^{(k)}>r^n)$ is summable, which by the Borel-Cantelli lemma implies \begin{equation}\label{dnkslrniop}{\sf GW}^{(k)}\Big(\{D_{n,k}^*<cn\}\cap\{Z_n^{(k)}>r^n\}\ \text{i.o.}\Big)=0.\end{equation}Now letting $\tilde{Z}_k$ represent the offspring distribution associated with $\tilde{{\bf T}}^{(k)}$ (conditioned on the event $\mathscr{A}_k$), we denote $$\tilde{\mu}_k:={\bf E}[\tilde{Z}_k]=\sum_{j=1}^k {\bf P}(Z=k)\cdot k.$$Noting that by the Kesten-Stigum Theorem $\frac{Z_n^{(k)}}{\tilde{\mu}_k^n}$ converges to a positive value ${\sf GW}^{(k)}$-a.s., we see that for any $r<\tilde{\mu}_k$, ${\sf GW}^{(k)}(Z_n^{(k)}\leq r^n\ \text{i.o.})=0$.  Combining this with \eqref{dnkslrniop}, along with the fact that $D_{n,k}\geq D_{n,k}^{*}$, we can conclude that \begin{equation}\label{dnknratigwas}\underset{n\to\infty}{\text{liminf}}\ \frac{D_{n,k}}{\sigma n}\geq\frac{\text{log}(1/\tilde{\mu}_k)}{\text{log}(1/\mu)}\ \ {\sf GW}^{(k)}-\text{a.s.}\end{equation}Since the expression on the right side of the above inequality goes to $1$ as $k\to\infty$, the proof is now complete.
\end{proof}

\medskip
\noindent
$Remark\ 1.$ A few different subtle variants of the above argument using independent subtrees, bounds on the growth rate of ${\bf T}$, and the Borel-Cantelli Lemma, will be employed at several other points throughout the paper in order to achieve similar results.  Instances of this include the proofs of Lemmas \ref{lemma:mdmc0c} and \ref{lemma:maxdepthreg}, as well as the beginning of the proof of Proposition \ref{prop:3ineqforlb}.

\bigskip
The next proposition that we will prove gives an asymptotic lower bound on the value of ${\sf SRW}_T(X_{2n}={\bf 0})$.  As noted in the introduction, to achieve this bound we estimate the probability that the walk begins by traveling to a vertex $v_n$ that belongs to a pipe with length of order $n^{1/3}$, then remains in this pipe for nearly the entire duration of the first $2n$ steps, and then travels straight back to the root, where it lands at time $2n$.  Approximating the probability of this sort of event will require the use of a result concerning the small deviation asymptotics for simple random walk on $\mathbb{Z}$.  This result, which appears as Theorem 3 in \cite{Mog} (and as Lemma 3.4 in \cite{GP}), will now be stated as a lemma.  Note that in the statement of the lemma, and throughout the rest of the paper, ${\sf SRW}_{\mathbb{Z}}$ will refer to the measure associated with simple random walk on $\mathbb{Z}$.

\medskip
\begin{lemma}\label{lemma:t3mogl34g}
Let $\underset{n\to\infty}{\emph{lim}}x(n)=\infty$ and $x(n)=o\big(\sqrt{n}\big)$.  Then, $$\underset{n\to\infty}{\emph{lim}}\frac{x(n)^2}{n}\emph{log}\bigg[{\sf SRW}_{\mathbb{Z}}\Big(\underset{j\leq n}{\emph{max}}|X_j|\leq x(n)\Big)\bigg]=-\frac{\pi^2}{8}.$$
\end{lemma}

\bigskip
\begin{prop}\label{prop:lbfrp}
Let $k$ satisfy ${\bf E}[Z\cdot 1_{Z\leq k}]>1$.  Then there exists $C_k>0$ such that for ${\sf GW}^{(k)}-\emph{a.s.}$ every tree $T$,$$\underset{n\to\infty}{\emph{liminf}}\ \frac{{\sf SRW}_T(X_{2n}={\bf 0})}{e^{-C_k n^{1/3}}}=\infty.$$
\end{prop}

\begin{proof}
Let $T$ be a tree with $T\in \mathscr{A}_k$ and where \begin{equation}\label{dnkonlilb}\underset{n\to\infty}{\text{liminf}}\ \frac{D_{n,k}(T)}{n}\geq \frac{\text{log}(1/	\tilde{\mu}_k)}{\text{log}\ \rho}.\end{equation}Now let $r_k:=\frac{\text{log}(1/	\tilde{\mu}_k)}{2\text{log}\ \rho}$ and note that by \eqref{dnkonlilb}, there must exist a value $N$ such that for each $n\geq N$ there is a vertex $v_n\in T$ for which $|v_n|<n^{1/3}$, $\text{deg}(u)\leq k$ for each $u<v$, and where $d_T(v_n)\geq r_k n^{1/3}$.  Next we define the events $B_{n,1}$, $B_{n,2}$, and $B_{n,3}$ as follows: $B_{n,1}$ is defined to be the event where $\{X_j\}$ takes its first $|v_n|$ steps towards $v_n$ (so that it lands on $v_n$ at time $t=|v_n|$), and then takes another $\lceil{\frac{r_k}{2}n^{1/3}\rceil}$ steps away from the root.  $B_{n,2}$ is defined as the event where $\{X_j\}$ resides at an $i$th generation descendant of $v_n$ at time $t=|v_n|+\lceil{\frac{r_k}{2}n^{1/3}\rceil}$ (for $1\leq i\leq 2\lceil{\frac{r_k}{2}n^{1/3}\rceil}-2$), and then remains among the first $2\lceil{\frac{r_k}{2}n^{1/3}\rceil}-2$ generations of proper descendants of $v_n$ until the first time $t$ when $|X_{2n-t}|=t$. Finally, $B_{n,3}$ is simply defined to be the event $\{X_{2n}={\bf 0}\}$.  We'll now use these three events to obtain a lower bound on ${\sf SRW}_T(X_{2n}={\bf 0})$ by first noting that \begin{equation}\label{bndingp0wbns}{\sf SRW}_T(X_{2n}={\bf 0})\geq{\sf SRW}_T\Big(B_{n,1}\cap B_{n,2}\cap B_{n,3}\Big)={\sf SRW}_T(B_{n,1})\cdot{\sf SRW}_T(B_{n,2}|B_{n,1})\cdot{\sf SRW}_T(B_{n,3}|\underset{i\leq 2}{\cap} B_{n,i}).\end{equation}Based on the definition of $v_n$, we obtain the bounds ${\sf SRW}_T(B_{n,1})\geq e^{-\text{log}(k+1)(1+r_k)n^{1/3}}$, ${\sf SRW}_T(B_{n,2}|B_{n,1})\geq {\sf SRW}_{\mathbb{Z}}\Big(\underset{j\leq 2n}{\text{max}}|X_j|\leq\lceil{\frac{r_k}{2}n^{1/3}\rceil}-1\Big)$, and ${\sf SRW}_T(B_{n,3}|\underset{i\leq 2}{\cap}B_{n,i})\geq e^{-\text{log}(k+1)(1+2r_k)n^{1/3}}$.  Observing that it follows from Lemma \ref{lemma:t3mogl34g} that \begin{equation}\label{SRWzprobconf}{\sf SRW}_{\mathbb{Z}}\Big(\underset{j\leq 2n}{\text{max}}|X_j|\leq\lceil{\frac{r_k}{2}n^{1/3}\rceil}-1\Big)=e^{-\frac{\pi^2}{r_k^2}n^{1/3}(1+o(1))},\end{equation}and combining this with \eqref{bndingp0wbns} and the bounds for ${\sf SRW}_T(B_{n,1})$ and ${\sf SRW}_T(B_{n,3}|\underset{i\leq 2}{\cap}B_{n,i})$, we now find that for $C_k:=\text{log}(k+1)\cdot(2+3r_k)+\frac{10}{r_k^2}$, we have $$\underset{n\to\infty}{\text{liminf}}\ \frac{{\sf SRW}_T(X_{2n}={\bf 0})}{e^{-C_k n^{1/3}}}=\infty.$$Since our only assumption about $T$ (in addition to being in $\mathscr{A}_k$) was \eqref{dnkonlilb}, which as we saw in \eqref{dnknratigwas} holds for almost surely every $T\in \mathscr{A}_k$, the proof is now complete.
\end{proof}

\medskip
The proof of Theorem \ref{theorem:diffbndp1p} will ultimately be broken up into two main parts, corresponding to the two inequalities in \eqref{limliminftsin}.  The first part, which is the simpler of the two, will consist of showing that \begin{equation}\label{limliminf1sin1}\underset{A\to\infty}{\text{lim}}\Bigg[\underset{n\to\infty}{\text{liminf}}\ {\sf SRW}_{\bf{T}}\bigg(\underset{j\leq 2n}{\text{max}}|X_j|\geq\frac{1}{A}n^{1/3}\ \Big|\ X_{2n}={\bf 0}\bigg)\Bigg]=1\ \ {\sf GW}-\text{a.s.}\end{equation}and the second will be to show that \begin{equation}\label{limliminf1sin2}\underset{A\to\infty}{\text{lim}}\Bigg[\underset{n\to\infty}{\text{liminf}}\ {\sf SRW}_{\bf{T}}\bigg(\underset{j\leq 2n}{\text{max}}|X_j|\leq An^{1/3}\ \Big|\ X_{2n}={\bf 0}\bigg)\Bigg]=1\ \ {\sf GW}-\text{a.s.}\end{equation}Establishing \eqref{limliminf1sin1} will not require much beyond Lemma \ref{lemma:maxdepthline} and Proposition \ref{prop:lbfrp}.  To prove \eqref{limliminf1sin2} however, we will first need to achieve several additional results.

\medskip
\begin{lemma}\label{lemma:profdegstov}
For any vertex $v$ in a tree $T$, define $D_T(v):=\underset{v'<v}{\prod}\emph{deg}(v')$, and set $M_T(n):=\emph{max}\{D_T(v):|v|=n\}$.  Then there exists $\alpha>0$ such that for ${\sf GW}-\emph{a.e.}$ tree $T$, there exists $N$ (which can depend on $T$) such that $M_T(n)\leq \alpha^n$ for all $n\geq N$.
\end{lemma}

\begin{proof}
For any tree $T$, let $T_n$ represent $T$ up through level $n$, let $Z_n(T)$ denote the size of the $n$th generation of $T$, and let $W_n(T):=\frac{Z_n(T)}{\mu^n}$.  In addition, define ${\sf GW}_n:=W_n\cdot{\sf GW}$.  Finally, let $\mu_n$ represent the uniform measure on non-backtracking paths of length $n$ in $T$ starting at the root, and define the measure ${\sf UNIF}_n:={\sf GW}_n\times\mu_n$ on tuples $(T,\omega_n)\in\mathcal{T}_{\mathcal{I}}\times\Omega_n$, where $\mathcal{T}_{\mathcal{I}}$ represents the set of infinite rooted trees with no leaves, and $\Omega_n$ represents the non-backtracking paths to level $n$ (note that for ease of notation, we've suppressed any reference to $T$ in the symbols $\Omega_n$ and $\mu_n$, even though both are always defined with respect to a particular tree $T$).  Now for any path $\omega_n$, let $v_0,v_1\dots,v_n$ represent the vertices of $\omega_n$ starting with the root and ordered by height.  We wish to show that for any  sequence of positive integers $r_0,r_1\dots,r_{n-1}$, we have \begin{equation}\label{jntdistprdeg}{\sf UNIF}_n\big(\text{deg}(v_i)=r_i\ \forall\ i\ \text{s.t. }0\leq i\leq n-1\big)=\frac{\underset{0\leq i\leq n-1}{\prod}{\bf P}(Z=r_i)\cdot r_i}{\mu^n}.\end{equation}Letting $\partial T_n$ represent the level $n$ vertices of $T_n$, we first note that \begin{align*}{\sf UNIF}_n\big(\text{deg}(v_i)=r_i\ \forall\ i\ \text{s.t. }0\leq i\leq n-1\big)&=\int\frac{\big|\{\omega_n\in\partial{\bf T}_n:\text{deg}(v_i)=r_i\ \forall\ i\}\big|}{Z_n({\bf T})}d{\sf GW}_n \\ &=\int\frac{\big|\{\omega_n\in\partial{\bf T}_n:\text{deg}(v_i)=r_i\ \forall\ i\}\big|}{Z_n({\bf T})}\cdot W_n({\bf T})d{\sf GW} \\ &=\frac{1}{\mu^n}{\bf E}_{{\sf GW}}\big[\big|\{\omega_n\in\partial{\bf T}_n:\text{deg}(v_i)=r_i\ \forall\ i\}\big|\big],\end{align*}which implies that in order to establish \eqref{jntdistprdeg}, it will suffice to show that \begin{equation}\label{exnumdegeprod}{\bf E}_{{\sf GW}}\big[\big|\{\omega_n\in\partial{\bf T}_n:\text{deg}(v_i)=r_i\ \forall\ i\}\big|\big]=\underset{0\leq i\leq n-1}{\prod}{\bf P}(Z=r_i)\cdot r_i.\end{equation}The case of $n=1$ is immediate.  If we now assume that \eqref{exnumdegeprod} holds for all $j<n$ (for $n\geq 2$), then we find that \begin{align*}{\bf E}_{{\sf GW}}\big[\big|\{\omega_n\in\partial{\bf T}_n:\text{deg}(v_i)=r_i\ \forall\ i\}\big|\big]&={\bf P}(Z=r_0)\cdot r_0\cdot{\bf E}_{{\sf GW}}\big[\big|\{\omega_{n-1}\in\partial{\bf T}_{n-1}:\text{deg}(v_i)=r_{i+1}\ \forall\ i\}\big|\big] \\ &=\underset{0\leq i\leq n-1}{\prod}{\bf P}(Z=r_i)\cdot r_i.\end{align*}Hence, \eqref{exnumdegeprod}, and therefore \eqref{jntdistprdeg}, now follows by induction.

In a slight abuse of notation, we now let $D_T(\omega_n)$ represent $D_T(v_n)$, and observe that \eqref{jntdistprdeg} implies that $${\bf E}_{{\sf UNIF}_n}[D_{{\bf T}}(\omega_n)^{\delta}]=\frac{1}{\mu^n}\sum_{r_0,\dots,r_{n-1}}\Big({\bf P}(Z=r_0)r_0\Big)\dots\Big({\bf P}(Z=r_{n-1})r_{n-1}\Big)\big(r_0\cdot r_1\dots r_{n-1}\big)^{\delta}=\frac{1}{\mu^n}{\bf E}[Z^{1+\delta}]^n.$$ By Markov's inequality, it follows that for $\alpha>0$, we have \begin{equation}\label{bndonpdtgtcn}{\sf UNIF}_n\Big(D_{{\bf T}}(\omega_n)>\alpha^n\Big)\leq\bigg(\frac{{\bf E}[Z^{1+\delta}]}{\mu\cdot \alpha^{\delta}}\bigg)^n.\end{equation}Combining this with the fact that $${\sf UNIF}_n\Big(D_{{\bf T}}(\omega_n)>\alpha^n\Big)\geq\int\frac{1}{Z_n({\bf T})}1_{\{M_{{\bf T}}(n)>\alpha^n\}}d{\sf GW}_n=\int\frac{W_n({\bf T})}{Z_n({\bf T})}1_{\{M_{{\bf T}}(n)>\alpha^n\}}d{\sf GW}=\frac{{\sf GW}(M_{{\bf T}}(n)>\alpha^n)}{\mu^n},$$we find that ${\sf GW}(M_{{\bf T}}(n)>\alpha^n)\leq\Big(\frac{{\bf E}[Z^{1+\delta}]}{\alpha^{\delta}}\Big)^n$.  Taking $\alpha>{\bf E}[Z^{1+\delta}]^{1/\delta}$, the lemma now follows from the Borel-Cantelli lemma.
\end{proof}

\medskip
Our first application of this last result will be in the proof of the following key proposition.  Throughout the proof, the symbols $\Omega_n$ and $\omega_n$ will refer to the same objects as in the proof of Lemma \ref{lemma:profdegstov}, except that the trees on which they're defined will have weighted edges.

\medskip
\begin{prop}\label{prop:lbo1cir}
For any tree $T\in\mathcal{T}_{\mathcal{I}}$, let $v$ be a vertex in $T_n$, and define $$m_T(v):=\big|\{v'<v:\emph{deg}(v')\geq 2\}\big|.$$Then there exists $c>0$ such that for ${\sf GW}-\emph{a.s.}$ every $T\in\mathcal{T}_{\mathcal{I}}$ there is an $N$ (which can depend on $T$) such that for every $n\geq N$ we have $\underset{v\in T_n}{\emph{min}}m_T(v)\geq cn$.
\end{prop}

\begin{proof}
We start by defining $\mathcal{T}_w$ to be the set of infinite rooted trees with no leaves or degree-$1$ vertices, endowed with positive integer edge weights (a tree in $\mathcal{T}_w$ will be denoted as $T_w$).  Now let $\phi:\mathcal{T}_{\mathcal{I}}\to\mathcal{T}_w$ be the map which acts on a tree $T$ by collapsing each maximal chain of edges for which all interior vertices have degree $1$ to a single edge, and then assigning each of these edges an edge weight equal to the length of the original chain (an edge that does not border a degree-$1$ vertex remains as is and is assigned the value $1$).  Letting $\mathcal{F}_w$ be the canonical $\sigma$-field on $\mathcal{T}_w$ generated by both the tree and its corresponding edge weights, we define the measure ${\sf GW}^*$ on $(\mathcal{T}_w,\mathcal{F}_w)$ as ${\sf GW}^*(\cdot)={\sf GW}'(\phi^{-1}(\cdot))$ (where ${\sf GW}'$ denotes ${\sf GW}(\cdot\ |\ \text{deg}({\bf 0})\geq 2)$), and then note that ${\sf GW}^*$ corresponds to a Galton-Watson process with offspring distribution $Z^*:=(Z|Z\geq 2)$, where edges are assigned weights that are i.i.d. geometric random variables with parameter $p:={\bf P}(Z=1)$.  Next we define the function $S^{(n)}:\mathcal{T}_w\times\Omega_n\to\mathbb{R}$ so that $S^{(n)}(\omega_n)$ (for $\omega_n\in\Omega_n$) is equal to the sum of the weights of the edges in $\omega_n$.  Noting that the edge weights of $\omega_n$ (presuming $(T_w,\omega_n)$ is selected according to the product measure ${\sf GW}^*\times\mu_n$, denoted below as ${\sf GWP}^*_n$) are i.i.d. geometric r.v.'s with parameter $p$, it follows from a basic application of large deviation theory (see, for instance, \cite{DZ}) that there exists a continuous increasing function $\Psi:(\frac{1}{1-p},\infty)\to (0,\infty)$ such that for every $n$, ${\sf GWP}^*_n\Big(\frac{S^{(n)}(\omega_n)}{n}>r\Big)\leq e^{-\Psi(r)n}$ (where $\underset{x\to\infty}{\text{lim}}\Psi(x)=\infty$). 

For the next step, we start by observing that the proof of Lemma \ref{lemma:profdegstov} above works for any offspring distribution and corresponding Galton-Watson measure, provided the offspring distribution is greater than $1$ almost surely and has $1+\delta$ moments.  In particular, this means that there exists an $\alpha<\infty$ for which the conclusion of Lemma \ref{lemma:profdegstov} holds for the offspring distribution $Z^*$.  Selecting some $r<\infty$ such that $\Psi(r)>\text{log}\ \alpha$, we then observe that \begin{equation}\label{intbndfsbnd}\int 1_{\{S^{(n)}(\omega_n)>rn\}}d{\sf GW}^*\times\mu_n\leq e^{-\Psi(r)n}\ \ \forall\ n\geq 1.\end{equation}Now defining $U_N:=\{(T_w,\omega_N)\in\mathcal{T}_w\times\Omega_N:M_{T_w}(n)\leq \alpha^n\ \forall\ n\geq N\}$, we note that \eqref{intbndfsbnd} implies that for any $N$, we have \begin{equation}\label{intbndfsbndA}\int_{U_N} 1_{\{S^{(n)}(\omega_n)>rn\}}d{\sf GW}^*\times\mu_n\leq e^{-\Psi(r)n}\ \ \forall\ n\geq N.\end{equation}Since $\mu_n(\omega_n)=\frac{1}{D_{T_w}(v_n)}$ for any $(T_w,\omega_n)\in\mathcal{T}_w\times\Omega_n$, it follows from \eqref{intbndfsbndA} and the definition of $U_N$, that for any $n\geq N$ we have$${\sf GW}^*\Big(\{T_w\in\mathcal{T}_w:\underset{\omega_n\in\Omega_n}{\text{max}}S^{(n)}(\omega_n)>rn\}\cap U_N\Big)\leq \alpha^n e^{-\Psi(r)n}=e^{-(\Psi(r)-\text{log}\ \alpha)n}.$$Since $r$ was chosen to satisfy $\Psi(r)>\text{log}\ \alpha$, the expression on the right of the inequality is summable, which means by Borel-Cantelli, for ${\sf GW}^*$-a.s. every $T_w\in U_N$, there exists $N'$ such that $\underset{\omega_n\in\Omega_n(T_w)}{\text{max}}S^{(n)}(\omega_n)\leq rn$ for all $n\geq N'$.  Since the $U_n$'s are increasing, and because it follows from Lemma \ref{lemma:profdegstov} that ${\sf GW}^*\big(\underset{n}{\cup}U_n\big)=1$, we can further conclude that for ${\sf GW}^*$-a.s. every weighted tree $T_w$, there exists a value $N$ (that can depend on $T_w$) such that $\underset{\omega_n\in\Omega_n(T_w)}{\text{max}}S^{(n)}(\omega_n)\leq rn$ for all $n\geq N$.  Choosing any value for $c$ (defined in the statement of the proposition) that satisfies $0<c<\frac{1}{r}$, the result now follows.
\end{proof}

\medskip
With Proposition \ref{prop:lbo1cir} established, we will now use it to achieve the main preliminary result that will be needed in order to prove \eqref{limliminf1sin2}.  In words, what this result states is that there is a value $p>0$ such that ${\sf GW}-$a.s. every $T$ has the property that for all vertices $v\in T$ that are sufficiently far from ${\bf 0}$, the percentage of vertices $v'$ along $\omega_v$ (the path connecting ${\bf 0}$ to $v$) for which, with probability at least $p$, simple random walk beginning at $v'$ never returns to $\omega_v$ following its first step, is bounded away from $0$.  The importance of this conclusion lies in the fact that it immediately implies the existance of an almost sure exponential upper bound on the probability that simple random walk, beginning sufficiently far from ${\bf 0}$, ever reaches ${\bf 0}$.

\medskip
\begin{prop}\label{prop:lbonpolt}
Take $T\in\mathcal{T}_{\mathcal{I}}$, let $v\in T$ be a non-root vertex, let $v'\in T$ be a child of $v$, and define $$\hat{P}_T(v,v'):={\sf SRW}_T\Big(\{|X_1|=|v|+1\}\cap\{X_1\neq v'\}\cap\{X_j\neq v\ \forall\ j\geq 1\}\ \Big|\ X_0=v\Big).$$In addition, let ${\bf 0}=v_0,v_1,\dots,v_n=v$ represent the vertices along the path from the root to $v$, and set $$N_p(T,v):=\big|\{1\leq j<n:\hat{P}_T(v_j,v_{j+1})\geq p\}\big|.$$Then there exists $p>0$, $s>0$ such that for ${\sf GW}-\emph{a.e.}\ T$, there is an $N$ (which can depend on $T$) such that for all $v\in T$ with $|v|\geq N$ we have $N_p(T,v)\geq s|v|$.
\end{prop}

\begin{proof}
To prove this result, we sample from the space $\mathcal{T}_{\mathcal{I}}\times\Omega_n$ according to the product measure ${\sf GW}\times\mu_n$.  An important property of the measure ${\sf GW}\times\mu_n$ that we will use is that when we sample according to it, the trees ${\bf T}(v_j)\setminus {\bf T}(v_{j+1})$ (for $1\leq j<n$) are independent and identically distributed (note that when $v_j$ has only one child, $T(v_j)\setminus T(v_{j+1})$ is a finite tree consisting of a single vertex).  This fact is illustrated by observing that sampling according to ${\sf GW}\times\mu_n$ can be achieved via the following random algorithm (note that to describe the algorithm we use a format resembling pseudocode):

\bigskip
\noindent
STEP 1: Let $T_0$ represent a tree consisting of a single vertex $v_0$.  Define an index $i$ with initial value $0$.

\medskip
\noindent
STEP 2: Employ the following loop:

\medskip
\noindent
While ($i<n$);

\medskip
Attach $Z$ child vertices to $v_i$;

\medskip
Pick one of the child vertices of $v_i$ uniformly at random and designate it $v_{i+1}$; 

\medskip
Attach trees with independent ${\sf GW}$ distributions to each of the children of $v_i$ other than $v_{i+1}$;

\medskip
Set $i=i+1$;

\medskip
\noindent
STEP 3: Once the loop terminates, attach a weighted tree to $v_n$ with distribution ${\sf GW}$.

\bigskip
\noindent
If we now let $({\bf T},\omega_n)$ be the tuple we obtain by setting ${\bf T}$ equal to the tree generated  through the above procedure, and $\omega_n$ equal to $(v_0,v_1,\dots,v_n)$, then it is not hard to see that the trees ${\bf T}(v_j)\setminus{\bf T}(v_{j+1})$ are i.i.d., and that $({\bf T},\omega_n)$ has distribution ${\sf GW}\times\mu_n$.

For the next step, let $T$, $v$, and $v'$ be as defined in the statement of the proposition, and define the event $A_p(T,v,v'):=\{0<\hat{P}_T(v,v')<p\}$.  Since each of the functions $1_{A_p({\bf T},v_j,v_{j+1})}$ depends entirely on the tree $T(v_j)\setminus T(v_{j+1})$, they are i.i.d.  Hence, if ${\sf GWP}_n\big(A_p({\bf T},v_j,v_{j+1})\big)$ is less than $\frac{r}{2}$ (where $r$ is defined as it was in the proof of Proposition \ref{prop:lbo1cir} and ${\sf GWP}_n:={\sf GW}\times\mu_n$) and we let $$\Psi_p\Big(\frac{r}{2}\Big):=\underset{\lambda}{\text{sup}}\ \frac{r}{2}\lambda-\text{log}\Big({\bf E}_{{\sf GWP}_n}\big[e^{-\lambda 1_{A_p({\bf T},v_1,v_2)}}\big]\Big),$$ then by another application of large deviations we have that \begin{equation}\label{bndsapld}{\sf GWP}_n\bigg(\frac{\sum_{j=1}^{n-1}1_{A_p({\bf T},v_j,v_{j+1})}}{n-1}\geq\frac{r}{2}\bigg)\leq e^{-\Psi_p(\frac{r}{2})(n-1)}.\end{equation}Noting that ${\sf GWP}_n\big(A_p({\bf T},v_1,v_2)\big)\to 0$ as $p\to 0$, we see that $\Psi_p\Big(\frac{r}{2}\Big)\to\infty$ as $p\to 0$.  Therefore, we can select $p>0$ such that $\Psi_p\Big(\frac{r}{2}\Big)>\text{log}\ \alpha$ (with $\alpha$ defined as in the proof of Proposition \ref{prop:lbo1cir}).  Letting $B({\bf T},n,p,r)$ represent the event $\frac{\sum_{j=1}^{n-1}1_{A_p({\bf T},v_j,v_{j+1})}}{n-1}\geq\frac{r}{2}$, it follows that \begin{equation}\label{intexpforpbn}\int 1_{B({\bf T},n,p,r)}d{\sf GW}\times\mu_n\leq e^{-\Psi_p(\frac{r}{2})(n-1)}.\end{equation}Now if we apply to \eqref{intexpforpbn} the same argument involving Lemma \ref{lemma:profdegstov} and the Borel-Cantelli lemma that was applied at the end of the proof of the previous proposition, we find that for ${\sf GW}-\text{a.s.}$ every $T$, there exists $N'$ (possibly depending on $T$) such that for all $n\geq N'$ we have $\big|\{1\leq j<n:0<\hat{P}_T(v_j,v_{j+1})<p\}\big|<\frac{r}{2}(n-1)$ for all $v$ such that $|v|=n$.  Letting $m_T(v)$ and $N$ be as defined in Proposition \ref{prop:lbo1cir}, and observing that $$N_p(T,v)=m_T(v)-\big|\{1\leq j<n:0<\hat{P}_T(v_j,v_{j+1})<p\}\big|,$$we can now conclude that if we let $s=\frac{r}{2}$ and $N''=\text{max}\{N,N'\}$, then for ${\sf GW}-\text{a.s.}$ every $T$, we have $N_p(T,v)\geq s|v|$ for all $v$ with $|v|\geq N''$.  Hence, the proof is complete.
\end{proof}

\bigskip
\begin{proof}[Proof of Theorem \ref{theorem:diffbndp1p}]
To establish \eqref{limliminf1sin1} we couple simple random walk on $T$, where $T$ is an infinite tree with no leaves, with simple random walk on the non-negative integers as follows: Let $\{X^{(1)}_j\}$ and $\{X^{(2)}_j\}$ represent our coupled simple random walks on $T$ and $\{\mathbb{Z}\geq 0\}$ respectively.  If $X^{(2)}_n<|X^{(1)}_n|$, then the $(n+1)$th step of $\{X^{(2)}_j\}$ is independent of that of $\{X^{(1)}_j\}$.  If instead $X^{(2)}_n=|X^{(1)}_n|>0$, then if the $(n+1)$th step of $\{X^{(1)}_j\}$ is towards the root, we have the $(n+1)$th step of $\{X^{(2)}_j\}$ go towards $0$, and if the $(n+1)$th step of $\{X^{(1)}_j\}$ goes away from the root, then the $(n+1)$th step of $\{X^{(2)}_j\}$ goes away from $0$ with probability $\frac{\text{deg}(X^{(1)}_n)+1}{2\text{deg}(X^{(1)}_n)}$ and towards $0$ with probability $\frac{\text{deg}(X^{(1)}_n)-1}{2\text{deg}(X^{(1)}_n)}$.  Noting that in the case where $X^{(2)}_n=|X^{(1)}_n|>0$, $\{X^{(2)}_j\}$ travels away from the root on its next step with probability  $\frac{1}{2}$, and that the coupling guarantees that $X^{(2)}_j\leq |X^{(1)}_j|$ for all $j$, we see that for any $n$, $A$ we have \begin{equation}\label{compSRWozptSRWotnl}{\sf SRW}_T\Big(\underset{j\leq 2n}{\text{max}}|X_j|\leq \frac{1}{A}n^{1/3}\Big)\leq {\sf SRW}_{\mathbb{Z}\geq 0}\Big(\underset{j\leq 2n}{\text{max}}|X_j|\leq \frac{1}{A}n^{1/3}\Big)={\sf SRW}_{\mathbb{Z}}\Big(\underset{j\leq 2n}{\text{max}}|X_j|\leq \frac{1}{A}n^{1/3}\Big).\end{equation}Once again invoking Lemma \ref{lemma:t3mogl34g}, it now follows from \eqref{compSRWozptSRWotnl} that $${\sf SRW}_T\Big(\underset{j\leq 2n}{\text{max}}|X_j|\leq \frac{1}{A}n^{1/3}\Big)\leq e^{-\frac{\pi^2 A^2}{4}n^{1/3}(1+o(1))}.$$Combining this last inequality with Proposition \ref{prop:lbfrp}, we observe that since all $T$ generated by $Z$ are infinite and have no leaves, this means that if $A>\frac{2}{\pi}\sqrt{C_k}$, then \begin{equation}\label{maxn1o3cx0gwk}\underset{n\to\infty}{\text{lim}}{\sf SRW}_{{\bf T}}\Big(\underset{j\leq 2n}{\text{max}}|X_j|\leq\frac{1}{A}n^{1/3}\ \Big|\ X_{2n}={\bf 0}\Big)=0\ \ {\sf GW}^{(k)}-\text{a.s.}\end{equation}Finally, noting that because ${\bf T}^{(k)}$ is a Galton-Watson tree with offspring distribution $Z\cdot 1_{Z\leq k}$, and because $Z\cdot 1_{Z\leq k}$ converges in distribution to $Z$ (and thus its probability generating function converges pointwise to that of $Z$), it follows that \begin{equation}\label{embtbnddeg}\underset{k\to\infty}{\text{lim}}{\sf GW}(\mathscr{A}_k)=1.\end{equation}This observation, alongside \eqref{maxn1o3cx0gwk}, now completes the proof of \eqref{limliminf1sin1}.

To establish \eqref{limliminf1sin2}, and thus complete the proof of the theorem, start by letting $T$ be a tree in $\mathscr{A}_k$ (for some $k<\infty$ with ${\bf E}[Z\cdot 1_{Z\leq k}]>1$) that satisfies the equality in Proposition \ref{prop:lbfrp}, and that satisfies Proposition \ref{prop:lbonpolt} for some $p>0$, $s>0$, and $N\in\mathbb{N}$.  Now if we take $A>0$, $n>0$ satisfying $A n^{1/3}\geq N$, then for any $v\in T$ with $|v|\geq A n^{1/3}$, the probability that simple random walk beginning at $v$ ever returns to the root will be bounded above by $e^{s\text{log}(1-p)|v|}\leq e^{s\text{log}(1-p)A n^{1/3}}$.  Hence, for each $j<2n$ we have \begin{align}\label{impbndfr16}&{\sf SRW}_T\Big(\{|X_j|\geq A n^{1/3}\}\cap\{X_{2n}={\bf 0}\}\Big)\leq e^{s\text{log}(1-p)A n^{1/3}} \nonumber\\ &\implies{\sf SRW}_T\bigg(\{\underset{j\leq 2n}{\text{max}}|X_j|\geq A n^{1/3}\}\cap\{X_{2n}={\bf 0}\}\bigg)\leq 2n e^{s\text{log}(1-p)A n^{1/3}}.\end{align}Selecting $A$ large enough so that $|s\text{log}(1-p)|A>C_k$, it now follows from \eqref{impbndfr16} and our assumption that $T$ satisfies the equality in Proposition \ref{prop:lbfrp}, that \begin{equation}\label{limforpreq10}\underset{n\to\infty}{\text{lim}}\ {\sf SRW}_T\bigg(\underset{j\leq 2n}{\text{max}}|X_j|\geq A n^{1/3}\ \Big|\ X_{2n}={\bf 0}\bigg)=0.\end{equation}Noting that it follows from \eqref{maxn1o3cx0gwk} and Proposition \ref{prop:lbonpolt} that ${\sf GW}-\text{a.s.}$ every tree satisfies the conditions we imposed on $T$, we can now conclude that for ${\sf GW}-\text{a.s.}$ every $T$ there exists $A<\infty$ (which can depend on $T$) such that \eqref{limforpreq10} holds, thus establishing \eqref{limliminf1sin2} and completing the proof of the theorem. 
\end{proof}

\section{Case 1(b): ${\bf P}(Z=0)>0$}

\begin{theorem}\label{brbnd0c}
If the offspring distribution $Z$ is supercritical and satisfies ${\bf P}(Z=0)>0$, and there exists $\delta>0$ such that ${\bf E}[Z^{1+\delta}]<\infty$, then \begin{equation}\label{ubnd0c}\underset{A\to\infty}{\emph{lim}}\Bigg[\underset{n\to\infty}{\emph{liminf}}\ {\sf SRW}_{\bf{T}}\bigg(\frac{1}{A}n^{1/3}\leq\underset{j\leq 2n}{\emph{max}}|X_j|\leq An^{1/3}\ \Big|\ X_{2n}={\bf 0}\bigg)\Bigg]=1\ \ {\sf GW}-\emph{a.s.}\end{equation}conditioned on non-extinction.
\end{theorem}

\bigskip
\noindent
The proof of the upper bound in \eqref{ubnd0c} strongly resembles that of \eqref{limliminf1sin2}, and relies primarily on a pair of results (Lemma \ref{lemma:mdmc0c} and Proposition \ref{prop:lbfrp2}) which are the analogs of Lemma \ref{lemma:maxdepthline} and Proposition \ref{prop:lbfrp} for the case where ${\bf P}(Z=0)>0$.  By contrast, the proof of the lower bound turns out to be far more difficult than that of \eqref{limliminf1sin1} due to the existence of finite subtrees, and thus will require a considerable amount of additional work.  Elaborating slightly on the description of the proof given in the introduction, the key steps can be organized as follows: 

\begin{enumerate}
\item We begin with Lemmas \ref{lemma:SRWrtft} and \ref{lemma:rficem}, which examine the exponential moments of random walk return times on finite trees, as well as the sizes of such trees.

\item These results are then used to deal with perhaps the most challenging step, which is Proposition \ref{prop:3ineqforlb}, where we establish the existence of asymptotic bounds for three sums, each of which relates to the probability that the lower bound in \eqref{ubnd0c} is not satisfied.

\item From here the remainder of the proof is fairly straight forward, and consists mainly of showing that the events analyzed in \ref{prop:3ineqforlb} constitute, in essence, a deconstruction of the event that the lower bound in \eqref{ubnd0c} fails to apply.  Comparing this probability to the probability that the walk returns to the root at time $2n$ (approximated in Proposition \ref{prop:lbfrp2}) then allows us to complete the proof.
\end{enumerate}

In order to state the first two results that we'll prove in this section, which are the analogs of Lemma \ref{lemma:maxdepthline} and Proposition \ref{prop:lbfrp} referenced above, we introduce the following definitions: Let $m:=\text{min}\{j>0:{\bf P}(Z=j)>0\}$, and for any rooted tree $T$ and vertex $v\in T$ define $h_T(v)$ to be the length of the longest path $v=v_0,v_1,\dots,v_n$ going away from the root such that for every $j<n$ the vertex $v_j$ has exactly $m$ children (i.e.~deg$(v_j)=m$) and all of them except $v_{j+1}$ are leaves (when $j=n-1$ the vertex $v_{j+1}$ is also permitted to be a leaf).  In addition, for every $n\geq 0$ and $k\geq m$, let $H_{n,k}(T):=\text{max}\{h_T(v):|v|=n,\ \text{deg}(u)\leq k\ \forall\ u<v\}$.  Any additional terms in the following lemma and proof which also appeared in Lemma \ref{lemma:maxdepthline} are assumed to carry the same definitions as earlier, unless we state otherwise.

\medskip
\begin{lemma}\label{lemma:mdmc0c}
Let $\mathscr{A}_k$ be defined as in Lemma \ref{lemma:maxdepthline}, except trees in $\mathscr{A}_k$ are now permitted to have leaves, and define $\sigma:=\frac{\emph{log}(1/\mu)}{\emph{log}\ \rho}$, where this time $\rho:={\bf P}(Z=m)\cdot m\cdot{\bf P}(Z=0)^{m-1}$.  Then for every $\epsilon>0$, there exists a value $N_{\epsilon}$ such that for $k\geq N_{\epsilon}$ we have \begin{equation}\label{hnkrlb}\underset{n\to\infty}{\emph{liminf}}\frac{H_{n,k}({\bf T})}{\sigma n}\geq 1-\epsilon\ \ {\sf GW}-\emph{a.s}.\end{equation}conditioned on $\mathscr{A}_k$.
\end{lemma}

\begin{proof}
Setting $H^*_{n,k}=\text{max}\{h_T(v):|v|=n,\ v\in \tilde{T}^{(k)}\}$, and noting that ${\sf GW}^{(k)}(Z_1=m,\ Z_2=1)=\rho$, we observe that for $\ell>0$, \begin{equation}\label{hnkscbk}{\sf GW}^{(k)}\big(H^*_{n,k}<\ell\ |\ Z^{(k)}_n\big)=\Big(1-\rho^{\ell}\Big)^{Z^{(k)}_n}.\end{equation}As in the proof of Lemma \ref{lemma:maxdepthline}, we now observe that it follows from \eqref{hnkscbk} and the Borel-Cantelli lemma that for $r>1$ and $0<c<\frac{\text{log}(1/r)}{\text{log}\ \rho}$, we have \begin{equation}\label{pkhnkintk0}{\sf GW}^{(k)}\Big(\{H_{n,k}^*<cn\}\cap\{Z_n^{(k)}>r^n\}\ \text{i.o.}\Big)=0.\end{equation}Noting that it once again follows from the Kesten-Stigum Theorem that for $r<	\tilde{\mu}_k$ we have ${\sf GW}^{(k)}(Z_n^{(k)}\leq r^n\ \text{i.o.})=0$, \eqref{pkhnkintk0} and the fact that $H_{n,k}\geq H_{n,k}^*$ now allow us to conclude that \begin{equation}\label{limirhnklb}\underset{n\to\infty}{\text{liminf}}\ \frac{H_{n,k}}{\sigma n}\geq\frac{\text{log}(1/	\tilde{\mu}_k)}{\text{log}(1/\mu)}\ \ {\sf GW}^{(k)}-\text{a.s.},\end{equation}thus completing the proof.
\end{proof}

\medskip
The following proposition is an exact restatement of Proposition \ref{prop:lbfrp}.  However, it is now being proven for the case where ${\bf P}(Z=0)>0$.  Since the proof only requires slight modifications to that of Proposition \ref{prop:lbfrp}, a number of details will be omitted.

\medskip
\begin{prop}\label{prop:lbfrp2}
Let $k$ satisfy ${\bf E}[Z\cdot 1_{Z\leq k}]>1$.  Then there exists $C_k>0$ such that for ${\sf GW}^{(k)}-\emph{a.s.}$ every tree $T$,$$\underset{n\to\infty}{\emph{liminf}}\ \frac{{\sf SRW}_T(X_{2n}={\bf 0})}{e^{-C_k n^{1/3}}}=\infty.$$
\end{prop}

\begin{proof}
Let $T$ be a tree with $T\in \mathscr{A}_k$ and where \begin{equation}\label{lbhnknr}\underset{n\to\infty}{\text{liminf}}\ \frac{H_{n,k}(T)}{n}\geq \frac{\text{log}(1/	\tilde{\mu}_k)}{\text{log}\ \rho}\end{equation}(with $\rho$ defined as in Lemma \ref{lemma:mdmc0c}).  Now once again define $r_k:=\frac{\text{log}(1/	\tilde{\mu}_k)}{2\text{log}\ \rho}$ (now using the new definition of $\rho$) and observe that \eqref{limirhnklb} implies we can select a sequence of vertices $\{v_n\}$ in $T$ such that, for all $n$ sufficiently large, we have $|v_n|<n^{1/3}$, $\text{deg}(u)\leq k\ \forall\ u<v_n$, and $h_T(v_n)\geq r_k n^{1/3}$.  Next we define the events $B_{n,1}$, $B_{n,2}$, and $B_{n,3}$ in the same way as in the proof of Proposition \ref{prop:lbfrp} (except they are now defined in terms of the new definitions of $r_k$ and $\{v_n\}$).  The string of inequalities in \eqref{bndingp0wbns}, as well as the lower bounds for ${\sf SRW}_T(B_{n,1})$ and ${\sf SRW}_T(B_{n,3}|\underset{i\leq 2}{\cap}B_{n,i})$ given in the proof, clearly apply for the new $B_{n,i}$'s.  In addition, since every non-leaf vertex among the first $2\lceil{\frac{r_k}{2}\rceil}-2$ generations of descendants of $v_n$ has exactly one non-leaf child, the bound on ${\sf SRW}_T(B_{n,2}|B_{n,1})$ from the proof of Proposition \ref{prop:lbfrp} also continues to apply.  Hence, the proof is complete.
\end{proof}

\medskip
\begin{lemma}\label{lemma:SRWrtft}
Let $T$ be a finite rooted tree to which we attach additional vertices $v_1,\dots,v_n$, so that each is attached to the root of $T$ by an edge, and let $L$ represent the time it takes for simple random walk starting at ${\bf 0}$ to reach the set $\{v_1,\dots,v_n\}$.  If we let $|T|$ represent the number of vertices in $T$ (not including $v_1,\dots,v_n$), and we take $\lambda\leq \emph{min}\{\frac{n}{18|T|},\frac{n^2}{18|T|^2},\frac{1}{18|T(u_1)|^2},\dots,\frac{1}{18|T(u_m)|^2}\}$ (where $u_1,\dots,u_m$ represent the children of ${\bf 0}$), then \begin{equation}\label{ebndonrt}{\bf E}_{{\sf SRW}_T}\big[e^{\lambda L}\big]\leq e^{\lambda\big(\frac{5|T|}{n}+1\big)}.\end{equation}
\end{lemma}

\begin{proof}
Denoting ${\bf E}_{{\sf SRW}_T}\big[e^{\lambda L}\big]$ as $f_{T,n}(\lambda)$ (recall $n$ is the number of additional leaves we've added to the root of $T$), we'll start by addressing the case where $n=1$.  The expression $f_T(\lambda)$ will often be used as shorthand for $f_{T,1}(\lambda)$.  The result we'll prove for the $n=1$ case, which is in fact slightly stronger than \eqref{ebndonrt}, is \begin{equation}\label{lbndemfn1}f_T(\lambda)\leq e^{\lambda(4|T|-3)}.\end{equation}To establish this, we first note that when $T$ just consists of a single vertex, \eqref{lbndemfn1} is immediate.  If instead $T$ has height $1$ (with $m$ children) then since in this case $L=2X-1$ (where $X$ is a geometric random variable with success probability $p=\frac{1}{m+1}$), we find that $$f_T(\lambda)=\frac{p e^{\lambda}}{1-(1-p)e^{2\lambda}}=\frac{\frac{1}{m+1}e^{\lambda}}{1-\frac{m}{m+1}e^{2\lambda}}=\frac{e^{\lambda}}{1-m(e^{2\lambda}-1)}.$$Expressing $\lambda$ as $\frac{\alpha}{m^2}$ (note the assumption $\lambda\leq \frac{n^2}{18|T|^2}$ implies $\alpha<\frac{1}{18}$) and using the above formula for $f_T(\lambda)$, we now observe that \begin{equation}\label{inforbndonft}f_T(\lambda)=\frac{e^{\frac{\alpha}{m^2}}}{1-m\underset{j\geq 1}{\sum}\frac{(2\alpha/m^2)^j}{j!}}\leq\frac{e^{\frac{\alpha}{m^2}}}{1-(\frac{2\alpha}{m}+\frac{12\alpha^2}{5m^3})}\leq\frac{e^{\frac{\alpha}{m^2}}}{e^{-(\frac{2\alpha}{m}+\frac{\alpha}{m^2})}}=e^{2\lambda|T|}\end{equation}(where the two inequalities in the middle are a consequence of the fact that $\alpha<\frac{1}{18}$).  Since the expression on the right in \eqref{inforbndonft} is less than $e^{\lambda(4|T|-3)}$ if $|T|\geq 2$, it follows that \eqref{lbndemfn1} holds whenever $T$ has height $1$.

To finish establishing \eqref{lbndemfn1}, we will induct on the height of $T$.  To do this, we first observe that if $T$ has height at least $2$, then \begin{equation}\label{eqforftfhal2}f_T(\lambda)=\frac{e^{\lambda}}{m+1}+\frac{e^{\lambda}}{m+1}f_T(\lambda)\sum_{j=1}^m f_{T(u_j)}(\lambda)\implies f_T(\lambda)=\frac{\frac{e^{\lambda}}{m+1}}{1-\frac{e^{\lambda}}{m+1}\underset{1\leq j\leq m}{\sum}f_{T(u_j)}(\lambda)}.\end{equation}Note that for the equality on the right to be valid, it is also required that the denominator in the last expression be positive.  However, given our assumptions that $T$ has height at least $2$ and $\lambda\leq\frac{1}{18|T|^2}$, the denominator in question will automatically be positive provided $f_{T(u_1)},\dots,f_{T(u_m)}$ all satisfy \eqref{lbndemfn1}.  Therefore, this issue can be addressed implicitly, in conjunction with the induction step, by establishing that \begin{equation}\label{indstforn1c}\frac{\frac{e^{\lambda}}{m+1}}{1-\frac{e^{\lambda}}{m+1}\underset{1\leq j\leq m}{\sum}e^{\lambda(4|T(u_j)|-3)}}\leq e^{\lambda(4|T|-3)}.\end{equation}Setting $\mathcal{N}:=|T|-1$ and $\mathcal{N}_j:=|T(u_j)|-1$, we first observe that \eqref{indstforn1c} can be expressed in the form\begin{equation}\label{refmubfis}1-\sum_{j=1}^m\Big(e^{(4\mathcal{N}_j+2)\lambda}-1\Big)\geq e^{-4\mathcal{N}\lambda}.\end{equation}Now using the fact that $\mathcal{N}=m+\underset{1\leq j\leq m}{\sum}\mathcal{N}_j$, along with the fact that $4\mathcal{N}_j+2<4(\mathcal{N}+1)=4|T|$, we see that \begin{align}\label{lbonlsinq}1-\sum_{j=1}^m\Big(e^{(4\mathcal{N}_j+2)\lambda}-1\Big)&=1-\sum_{j=1}^m\sum_{i=1}^{\infty}\frac{\big((4\mathcal{N}_j+2)\lambda\big)^i}{i!}=1-4\mathcal{N}\lambda+2m\lambda-\lambda\sum_{j=1}^m\sum_{i=2}^{\infty}\frac{(4\mathcal{N}_j+2)^i}{i!}\lambda^{i-1} \\ &\geq 1-4\mathcal{N}\lambda+2m\lambda-\lambda\sum_{j=1}^m\sum_{i=2}^{\infty}\frac{(4|T|)^i}{i!}\Big(\frac{1}{18|T|^2}\Big)^{i-1}\geq 1-4\mathcal{N}\lambda+\frac{3}{2}m\lambda.\nonumber\end{align}Combining this with the fact that $$e^{-4\mathcal{N}\lambda}\leq 1-4\mathcal{N}\lambda+8\mathcal{N}^2\lambda^2\leq 1-4\mathcal{N}\lambda+\frac{4}{9}\lambda\leq 1-4\mathcal{N}\lambda+\frac{3}{2}m\lambda$$ now establishes \eqref{refmubfis} (which is equivalent to \eqref{indstforn1c}), thus finishing the induction step, and therefore completing the proof of \eqref{lbndemfn1}.

For the general case, the result is again immediate if $T$ consists of just the root.  For $\text{height}(T)\geq 1$, we once again allow $\mathcal{N}$ and $\mathcal{N}_j$ to represent $|T|-1$ and $|T(u_j)|-1$ respectively, and start by noting that if $n\geq 2$, then we have \begin{align}\label{ftfgc}f_{T,n}(\lambda)&=\frac{n}{n+m}e^{\lambda}+\frac{e^{\lambda}}{n+m}f_{T,n}(\lambda)\sum_{j=1}^m f_{T(u_j)}(\lambda) \\ &\implies f_{T,n}(\lambda)=\frac{\frac{n}{n+m}e^{\lambda}}{1-\frac{e^{\lambda}}{n+m}\underset{1\leq j\leq m}{\sum}f_{T(u_j)}(\lambda)}\leq\frac{e^{\lambda}}{1-\frac{1}{n}\underset{1\leq j\leq m}{\sum}(e^{\lambda(4\mathcal{N}_j+2)}-1)}\nonumber\end{align}(where the last inequality follows from substituting the expression on the right in \eqref{lbndemfn1} for $f_{T(u_j)}(\lambda)$, and where the denominator in the last expression being positive follows from the bounds we've imposed on $\lambda$).  This now implies that it will suffice to show that \begin{equation}\label{frfubeml}1-\frac{1}{n}\sum_{j=1}^m(e^{\lambda(4\mathcal{N}_j+2)}-1)\geq e^{-\frac{5\mathcal{N}\lambda}{n}}.\end{equation}Now if we use the string of inequalities in \eqref{lbonlsinq}, and the fact that $\lambda\leq\frac{1}{18|T(u_j)|^2}$ for each $j$, then we see that the expression on the left in \eqref{frfubeml} is bounded below by $1-\frac{4\mathcal{N}\lambda}{n}+\frac{3m\lambda}{2n}$.  If we then combine this with the fact that $$e^{-\frac{5\mathcal{N}\lambda}{n}}\leq 1-\frac{5\mathcal{N}\lambda}{n}+\frac{25 \mathcal{N}^2\lambda^2}{2 n^2}\leq 1-\frac{4\mathcal{N}\lambda}{n}$$(this follows from the fact that $\lambda\leq\frac{n}{18|T|}$), we see that \eqref{frfubeml} follows, thus completing the proof of the lemma.
\end{proof}

\medskip
In order to make use of Lemma \ref{lemma:SRWrtft}, we will need to obtain some estimates related to the sizes of the finite trees that the random walk $\{X_n\}$ encounters on the random tree ${\bf T}$.  Therefore, the next result we present will address this issue.  To state this result, we need to define some additional notation.  First, for any infinite rooted tree $T$, let $T^i$ represent the tree we obtain by eliminating all vertices in $T$ that do not lie on infinite non-backtracking paths from the root to infinity, as well as all edges that touch these vertices (i.e. $T^i$ is the tree obtained by chopping off all of the bushes of $T$).  Now for any vertex $v\in T^i$, let $Z^i(v)$ represent the number of children that $v$ has in $T^i$.  Additionally, for $v\in T^i$, let $T^f$ represent the finite subtree rooted at $v$ (note that if all children of $v$ in $T$ are also in $T^i$, then $T^f$ simply consists of the single vertex $v$), let $Z^f(v)$ represent the size of the  first generation of $T^f(v)$ (i.e. the number of children $v$ has that are not in $T^i$), and denote $r(v):=\frac{|T^f(v)|}{Z^i(v)}$.

\medskip
\begin{lemma}\label{lemma:rficem}
If we let $\mathscr{A}$ represent the event that ${\bf T}$ does not go extinct and define ${\sf GW}_{\mathscr{A}}:={\sf GW}(\cdot|\mathscr{A})$, then the value $r({\bf 0})$ has exponential moments with respect to ${\sf GW}_{\mathscr{A}}$.
\end{lemma}

\begin{proof}
Our first step will be to show that $|{\bf T}|$ has exponential moments with respect to the measure ${\sf GW}(\cdot|\mathscr{A}^c)$.  To do this, we start by noting that $Z'$ (the offspring distribution associated with ${\sf GW}(\cdot|\mathscr{A}^c)$) has probability generating function $h(x)=\frac{f(qx)}{q}$, where $q:={\sf GW}(\mathscr{A}^c)$ (see \cite{LP}, Proposition 5.28).  From this, we can then conclude that $Z'$ itself has exponential moments, and that $h'(1)<1$.  Now if we let $F_n(x)$ represent the probability generating function of $|{\bf T}_n|$ with respect to ${\sf GW}(\cdot|\mathscr{A}^c)$, then we find that $F_0(x)=x$ and, for every $n\geq 1$, $F_n(x)=x h(F_{n-1}(x))$.  Furthermore, because $Z'$ is subcritical, has exponential moments, and satisfies ${\bf P}(Z'\geq 2)>0$, it follows that there exists $x_o\in (1,\infty)$ where the quantity $\frac{x}{h(x)}$ attains its maximum value.  Now if we have a value $x$ such that $1<x<\frac{x_o}{h(x_o)}$, and we assume $F_{n-1}(x)<x_o$, then it will follow that $F_n(x)=x h(F_{n-1}(x))<\frac{x_o}{h(x_o)}\cdot h(x_o)=x_o$.  Since we know $x<\frac{x_o}{h(x_o)}<x_o$, it follows by induction that $F_n(x)<x_o\ \forall\ n\geq 1$.  Therefore, the function $F(x):=\underset{n\to\infty}{\text{lim}}F_n(x)$ (the generating function for $|\bf T|$ with respect to ${\sf GW}(\cdot|\mathscr{A}^c)$) is finite for $x<\frac{x_o}{h(x_o)}$, and thus $|{\bf T}|$ has exponential moments with respect to ${\sf GW}(\cdot|\mathscr{A}^c)$.

To complete the proof of the lemma, we now let $\mu_o={\bf E}[Z']$, and observe that for $n$ large enough so that $\frac{2\mu_o}{2\mu_o +n}<\frac{1-q}{2}$, we have \begin{equation}\label{ubrgwpgn}{\sf GW}_{\mathscr{A}}\big(r({\bf 0})\geq n\big)\leq {\sf GW}_{\mathscr{A}}\Big(Z^f>\frac{n}{2\mu_o}Z^i\Big)+{\sf GW}_{\mathscr{A}}\Big(|{\bf T}^f|\geq n Z^i\ \Big|\ Z^f=\lfloor{Z^i n/2\mu_o\rfloor}\Big)\end{equation}(where ${\bf T}^f$, $Z^f$, and $Z^i$ represent ${\bf T}^f({\bf 0})$, $Z^f({\bf 0})$, and $Z^i({\bf 0})$ respectively).  For the first term on the right in \eqref{ubrgwpgn}, we note that it is bounded above by $$\frac{1}{1-q}\sum_{j\geq 1+\lceil{n/\mu_o\rceil}}{\bf P}(Z=j){\sf GW}\Big(Z^i<\frac{2\mu_o}{n}Z^f\ \Big|\ Z=j\Big)\leq\frac{1}{1-q}\sum_{j\geq 1+\lceil{n/\mu_o\rceil}}{\bf P}(Z=j){\sf GW}\Big(Z^i<\frac{1-q}{2}Z\ \Big|\ Z=j\Big).$$Since the term on the right in the above inequality is clearly exponentially small in $n$, it now just remains to deal with the second term on the right in \eqref{ubrgwpgn}.  Noting that this expression is equal to $$\frac{\sum_{j=1}^{\infty}{\sf GW}\Big(\{|{\bf T}^f|\geq jn\}\cap\{Z^f=\lfloor{jn/2\mu_o\rfloor}\}\cap\{Z^i=j\}\Big)}{\sum_{j=1}^{\infty}{\sf GW}\Big(\{Z^f=\lfloor{jn/2\mu_o\rfloor}\}\cap\{Z^i=j\}\Big)},$$and that $|{\bf T}^f|=1+\underset{i\leq Z^f}{\sum}|{\bf T}(v_i)|$ (where the $v_i$'s represent the children of the root in ${\bf T}^f$), we see that it can be bounded above by \begin{equation}\label{bnd2ndtes}\underset{j\geq 1}{\text{max}}\ {\sf GW}\Big(|{\bf T}^f|\geq jn\ \Big|\ Z^f=\lfloor{jn/2\mu_o\rfloor}\Big)\leq\underset{j\geq 1}{\text{max}}\ {\bf P}\Bigg(\sum_{i\leq\lfloor{jn/2\mu_o\rfloor}}|{\bf T}_i|\geq j(n-1)\Bigg)\end{equation}(where the ${\bf T}_i$'s are i.i.d. with distribution ${\sf GW}(\cdot|\mathscr{A}^c)$).  Having already determined that $|{\bf T}|$ has exponential moments (with respect to ${\sf GW}(\cdot|\mathscr{A}^c))$, it follows from a basic application of large deviations that the expression on the right in \eqref{bnd2ndtes} is exponentially small in $n$.  Hence, this addresses the second term on the right in \eqref{ubrgwpgn}, and thus completes the proof.
\end{proof}

\medskip
\begin{lemma}\label{lemma:bndnravutn}
Let $T$ be an infinite rooted tree with no leaves, and let $v_o$ represent the first vertex in $T$ (i.e. the vertex closest to the root) where $T$ branches (it is assumed $T$ is not the infinite half-line).  Then there exists a constant $C_o>0$ (independent of $T$) such that for all $C>C_o$ we have $${\sf SRW}_T\Big(\underset{v>v_o}{\emph{max}}\big|\{j\leq n:X_j=v\}\big|>Cn^{2/3}\Big)=o\Big(e^{-\frac{C}{5}n^{1/3}}\Big).$$
\end{lemma}

\begin{proof}
We begin by examining simple random walk on $\mathbb{Z}$, noting that the random variable consisting of the time of the first return to $0$ has probability generating function $f(x)=1-\sqrt{1-x^2}$.  Looking at the asymptotics of the coefficients of $f$, we can conclude that $${\sf SRW}_{\mathbb{Z}}\Big(\text{min}\{j>0:X_j=0\}\geq n\Big)=\frac{n^{-1/2}}{\sqrt{2\pi}}+o\big(n^{-1/2}\big),$$from which it follows that if we let $M_j$ represent the time of the $j$th return to $0$, then for any $C>0$ we have \begin{align}\label{bndr0SRWzlo}{\sf SRW}_{\mathbb{Z}}\Big(\big|\{j\leq n:X_j=0\}\big|>Cn^{2/3}\Big)&\leq{\sf SRW}_{\mathbb{Z}}\Big(\big|\{j\leq\lfloor{Cn^{2/3}\rfloor}:M_j-M_{j-1}\geq n^{2/3}\}\big|\leq n^{1/3}\Big)\\ &=o\bigg({\bf P}\Big[\text{Bin}\Big(\lfloor{Cn^{2/3}\rfloor},\frac{n^{-1/3}}{3}\Big)\leq n^{1/3}\Big]\bigg)=o\Big(e^{-\frac{C}{4}n^{1/3}}\Big)\nonumber\end{align}(where the last equality, which is obtained through the use of Stirling's formula, holds provided $C$ is large enough so that $\frac{C}{3}-\text{log}C-1+\text{log}3>\frac{C}{4}$).  Now in order to apply this to simple random walk on $T$, we construct a coupling between simple random walk on $\mathbb{Z}$ (beginning at $0$) and simple random walk on $T$ (beginning at a vertex $v>v_o$), that is similar to the coupling described near the beginning of the proof of Theorem \ref{theorem:diffbndp1p}.  Letting $\{X^{(1)}_n\}$ and $\{X^{(2)}_n\}$ represent these random walks on $T$ and $\mathbb{Z}$ respectively (with the specified starting positions), set $\mathcal{r}_j$ equal to the number of steps taken by $\{X^{(1)}_j\}$ after taking its $j$th step inside of $T(v_o)$, and define $\Upsilon_j:=X^{(1)}_{\mathcal{r}_j}$. Now for each $j$, if $|X^{(2)}_j|$ is less than the distance between $\Upsilon_j$ and $v$, then we have the $(j+1)$th step of $\{X^{(2)}_n\}$ be independent of that of $\{\Upsilon_n\}$.  If instead $|X^{(2)}_j|$ equals the distance between $\Upsilon_j$ and $v$ (and $X^{(2)}_j$ does not equal $0$), then if the $(j+1)$th step of $\{\Upsilon_n\}$ goes towards $v$, we have the $(j+1)$th step of $\{X^{(2)}_n\}$ go towards $0$, and if the $(j+1)$th step of $\{\Upsilon_n\}$ goes away from $v$, we have the $(j+1)$th step of $\{X^{(2)}_n\}$ go away from $0$ with probability $\frac{\text{deg}(\Upsilon_j)+1}{2\text{deg}(\Upsilon_j)}$, and towards $0$ with probability $\frac{\text{deg}(\Upsilon_j)-1}{2\text{deg}(\Upsilon_j)}$.  Since this coupling ensures that $X^{(2)}_j=0$ whenever $\Upsilon_j=v$, and since $\big|\{j\leq n:X^{(1)}_j=v\}\big|\leq \big|\{j\leq n:\Upsilon_j=v\}\big|$, it follows that for any vertex $v$ that is a descendent of $v_o$, the number of returns to $v$ up to time $n$ is dominated by the number of returns to $0$ for simple random walk on $\mathbb{Z}$ up to time $n$.  Moreover, since $$\underset{v>v_o}{\text{max}}\big|\{j\leq n:X_j=v\}\big|=\underset{j\leq n}{\text{max}}\big|\{j\leq i\leq n:X_i=X_j\}\big|,$$ it follows from applying \eqref{bndr0SRWzlo}, along with a union bound, that for $C$ large enough so that $\frac{C}{3}-\text{log}C-1+\text{log}3>\frac{C}{4}$, we have$${\sf SRW}_T\Big(\underset{v>v_o}{\text{max}}\big|\{j\leq n:X_j=v\}\big|>Cn^{2/3}\Big)=o\Big(n e^{-\frac{C}{4}n^{1/3}}\Big)=o\Big(e^{-\frac{C}{5}n^{1/3}}\Big),$$thus completing the proof.
\end{proof}

\medskip
Before stating the final proposition that will be needed in order to prove Theorem \ref{brbnd0c}, which involves establishing the asymptotic bounds described at the beginning of the section, we'll first need to provide several new definitions.  First, we define a natural coupling between simple random walk on $T$ and $T^i$ by letting $N_j$ represent the total number of steps taken by $\{X_n\}$ after taking its $j$th step inside of $T^i$, and then defining the process $\{Y_n\}$ as $Y_j:=X_{N_j}$.  Since the process $\{Y_n\}$ on $T$ has the same law as the process $\{X_n\}$ on $T^i$, each of the two types of notation will be used at different points throughout the remainder of the proof, with the particular choice generally depending on context.  Next we define $S(i):=\underset{0\leq t<i}{\sum}r(Y_t)$.  Finally, we set $\tilde{V}(T):={\bf 0}\cup\{v\in T^i:Z^i(v)\geq 2\}$ and define $W(i):=\big|\{j<i:Y_j\in\tilde{V}(T)\}\big|$.

\medskip
\begin{prop}\label{prop:3ineqforlb}
For any $q<\infty$, ${\sf GW}_{\mathscr{A}}-\emph{a.s.}$ every $T$ has the property that for $\delta>0$ sufficiently small,\begin{align}\label{3ineqgwas}&\sum_{i=1}^{\lfloor{\delta n\rfloor}}{\sf SRW}_T\Big(\underset{j<i}{\emph{max}}|Y_j|\leq \delta n^{1/3},\ N_i>2n,\ S(i)<\delta^{1/6}n\Big)=o\Big(e^{-qn^{1/3}}\Big)\\ \label{3ineqgwas2}&\sum_{i=1}^{\lfloor{\delta n\rfloor}}{\sf SRW}_T\bigg(\underset{j<i}{\emph{max}}|Y_j|\leq \delta n^{1/3},\ W(i)> n^{2/3}\bigg)=o\Big(e^{-qn^{1/3}}\Big)\\ \label{3ineqgwas3}&\sum_{i=1}^{\lfloor{\delta n\rfloor}}{\sf SRW}_T\bigg(\underset{j<i}{\emph{max}}|Y_j|\leq \delta n^{1/3},\ W(i)\leq n^{2/3},\ S(i)\geq \delta^{\frac{1}{6}}n\bigg)=o\Big(e^{-qn^{1/3}}\Big)\end{align}
\end{prop}

\begin{proof}
To establish \eqref{3ineqgwas}, we begin by looking at the quantity $r({\bf 0})$ from Lemma \ref{lemma:rficem}, and note that since it has exponential moments with respect to ${\sf GW}_{\mathscr{A}}$, it follows that there exists a constant $\alpha<1$ such that, for $n$ sufficiently large, ${\sf GW}_{\mathscr{A}}(r({\bf 0})\geq n)\leq \alpha^n$.  Now noting that ${\bf T}^i$, conditioned to survive, is itself a Galton-Watson tree with offspring distribution $(Z^i|Z^i\geq 1)$ (see Proposition 5.28, \cite{LP}), and then selecting some $M<\infty$ and $m>\mathbb{E}[Z^i|Z^i\geq 1]$, we find that for large enough $n$, we have \begin{align}{\sf GW}_{\mathscr{A}}\big(\text{max}\{r(v):v\in\partial{\bf T}^i_n\}\geq Mn\big)&\leq{\sf GW}_{\mathscr{A}}\big(Z^i_n\geq m^n\big)+1-\big(1-\alpha^{Mn}\big)^{m^n} \nonumber \\ &\leq\big(\mathbb{E}[Z^i|\mathscr{A}]/m\big)^n+1-\big(1-\alpha^{Mn}\big)^{m^n}\nonumber\end{align}(where $Z^i_n$ represents $|\partial{{\bf T}_n^i}|$).  Since the expression on the second line is summable for $M>\frac{\text{log}\ m}{\text{log}(1/\alpha)}$, it now follows from Borel-Cantelli that for $M>\frac{\text{log}\ m}{\text{log}(1/\alpha)}$, we have $${\sf GW}_{\mathscr{A}}\big(\text{max}\{r(v):v\in\partial{\bf T}^i_n\}\geq Mn\ \text{i.o.}\big)=0,$$which also implies that for ${\sf GW}_{\mathscr{A}}$ almost every $T$, \begin{equation}\label{lbrgwasem}\underset{|v|\leq n}{\text{max}}\ r(v)\leq Mn\end{equation}for $n$ sufficiently large.  In addition, if we let $s(v):=\underset{u\in T^f_1(v)}{\text{max}}|T^f(u)|$, then because $|{\bf T}|$ has exponential moments with respect to ${\sf GW}(\cdot|\mathscr{A}^c)$ (as shown in the first paragraph of the proof of Lemma \ref{lemma:rficem}), it follows (again by the same argument) that there exists $M'<\infty$ such that for ${\sf GW}_{\mathscr{A}}$ almost every $T$, \begin{equation}\label{lbrgwasem2}\underset{|v|\leq n}{\text{max}}\ s(v)\leq M'n\end{equation}for large $n$.  Thus, if we denote the event $$\left\{\underset{j<i}{\text{max}}\ r(Y_j)\leq M\delta n^{1/3}\right\}\cap\left\{\underset{j<i}{\text{max}}\ s(Y_j)\leq M'\delta n^{1/3}\right\}\cap\left\{S(i)<\delta^{\frac{1}{6}}n\right\}$$ as $B$, and let $Y_{[0,k]}=(Y_0,\dots,Y_k)$, then we find that for any $T$ satisfying \eqref{lbrgwasem} and \eqref{lbrgwasem2}, the summand in \eqref{3ineqgwas} is bounded above by $${\sf SRW}_T\Big(\left\{N_i>2n\right\}\cap B\Big)\leq\sum_{{\bf y}_{[0,i-1]}\in B}{\sf SRW}_T\Big({\bf y}_{[0,i-1]}\Big)\cdot{\sf SRW}_T\bigg(\sum_{j=1}^i (N_j-N_{j-1})>2n\ \bigg|\ {\bf y}_{[0,i-1]}\bigg)$$for all $i\leq \delta n$ for $n$ sufficently large.  Now denoting $M_o:=\text{max}\{M,M'\}$ and setting $\lambda=\frac{1}{18 \delta^2 M_o^2 n^{2/3}}$, we see that since the random variables $N_j-N_{j-1}$ are independent for distinct $j$ (when we are conditioning on a specific ${\bf y}_{[0,i-1]}=(y_0,\dots,y_{i-1})\in B$), it follows from Lemma \ref{lemma:SRWrtft} that $${\bf E}_{{\sf SRW}_T}\big[e^{\lambda N_i}\big|{\bf y}_{[0,i-1]}\big]=\prod_{j=1}^i {\bf E}_{{\sf SRW}_T}\big[e^{\lambda(N_j-N_{j-1})}\big|Y_{j-1}=y_{j-1}\big]\leq e^{\lambda\big(5\delta^{\frac{1}{6}}n+\delta n\big)}$$(where the $\delta n$ term at the end of the exponent follows from the fact that $i\leq \delta n$).  Using Markov's inequality, this then implies that $${\sf SRW}_T\bigg(\sum_{j=1}^i (N_j-N_{j-1})>2n\ \bigg|\ {\bf y}_{[0,i-1]}\bigg)\leq e^{\lambda\big(5\delta^{\frac{1}{6}}n+\delta n-2n\big)}.$$Now plugging this into the expression for the upper bound on ${\sf SRW}_T\Big(\left\{N_i>2n\right\}\cap B\Big)$, we find that the summand in \eqref{3ineqgwas} is bounded above by $e^{\lambda\big(5\delta^{\frac{1}{6}}n+\delta n-2n\big)}$ (for all $i\leq \delta n$ for $n$ sufficiently large).  Summing this over all $i\leq \delta n$ (i.e. multiplying it by $\lfloor{\delta n\rfloor}$), and recalling that $\lambda=\frac{1}{18\delta^2 M_o^2 n^{2/3}}$, we see that we can indeed obtain \eqref{3ineqgwas}  by taking $\delta$ sufficiently close to $0$.

Moving on to the proof of \eqref{3ineqgwas2}, we first observe that if we let $N(t):=\text{max}\{j:W(j)\leq t\}$, then for every $T\in\mathscr{A}$ we have \begin{align}\label{s2ndsi2s}&\sum_{i=1}^{\lfloor{\delta n\rfloor}}{\sf SRW}_T\bigg(\underset{j<i}{\text{max}}|Y_j|\leq \delta n^{1/3},\ W(i)> n^{2/3}\bigg) \\&\leq\sum_{i=1}^{\lfloor{\delta n\rfloor}}{\sf SRW}_T\bigg(\underset{j<i}{\text{max}}|Y_j|\leq \delta n^{1/3},\ W(i)>n^{2/3},\ \Big|\{j\leq N\big(n^{2/3}\big):Y_j\in\tilde{V}(T),\ |Y_{j+1}|=|Y_j|+1\}\Big|>\frac{7}{12}n^{2/3}\bigg) \nonumber \\ &+\sum_{i=1}^{\lfloor{\delta n\rfloor}}{\sf SRW}_T\bigg(\Big|\{j\leq N\big(n^{2/3}\big):Y_j\in\tilde{V}(T),\ |Y_{j+1}|=|Y_j|+1\}\Big|\leq\frac{7}{12}n^{2/3}\bigg). \nonumber\end{align}Since $\{Y_n\}$ performs simple random walk on $T^i$, and because every non-root vertex in $\tilde{V}(T)$ has, by definition, at least two children in $T^i$, it follows that for any $v\in\tilde{V}(T)$ we have ${\sf SRW}_T\big(|Y_{j+1}|=|Y_j|+1\ \big|\ Y_j=v\big)\geq\frac{2}{3}$.  Hence, using a large deviation bound, we can conclude that the summand on the third line in \eqref{s2ndsi2s} is bounded above by $e^{-Cn^{2/3}}$ (for some $C>0$ that is independent  of $T$), which implies that the sum itself is bounded above by $\delta n e^{-Cn^{2/3}}$.  

Now in order to obtain a corresponding upper bound for the sum on the second line, we start by defining $M(t):=\text{min}\{j:j+1-W(j+1)\geq t\}$ and $\Phi(i):=\big|\{j<i:Y_j\not\in \tilde{V}(T),\ |Y_{j+1}|=|Y_j|+1\}\big|-\big|\{j<i:Y_j\not\in \tilde{V}(T),\ |Y_{j+1}|=|Y_j|-1\}\big|$.  In words, $M(t)$ is a stopping time representing the number of steps $\{Y_n\}$ takes in $T^i$ until it has landed inside the set $\tilde{V}(T)^c$ (i.e. the set of non-root vertices with exactly one child in $T^i$) a total of $\lceil{t\rceil}$ times, and $\Phi(i)$ represents the net contribution to $|Y_i|$ that is made by steps taken by $\{Y_n\}$ before time $i$ from non-root vertices with exactly one child.  If we also define $\Phi'(i):=\big|\{j<i:Y_j\in \tilde{V}(T),\ |Y_{j+1}|=|Y_j|+1\}\big|-\big|\{j<i:Y_j\in \tilde{V}(T),\ |Y_{j+1}|=|Y_j|-1\}\big|$ (so that $\Phi'(i)$ represents the net contribution to $|Y_i|$ that is made by steps taken by $\{Y_n\}$ before time $i$ from vertices in $\tilde{V}(T)$), then we find that the event described inside the summand on the second line in \eqref{s2ndsi2s} implies that $\Phi'\big(N(n^{2/3})\big)>\frac{7}{12}n^{2/3}-\frac{5}{12}n^{2/3}-1=\frac{1}{6}n^{2/3}-1$.  Noting, in addition, that $W(i)>n^{2/3}$ implies $i>N(n^{2/3})$, and that $i\leq \delta n\leq M(\delta n)$, we see that if these observations are combined with the assumption $\underset{j<i}{\text{max}}|Y_j|\leq \delta n^{1/3}$, and the fact that $\Phi(j)+\Phi'(j)=|Y_j|$, then we can conclude that the event inside the summand implies that$$\underset{j\leq M(\delta n)}{\text{min}}\Phi(j)\leq\Phi\big(N(n^{2/3})\big)\leq\delta n^{1/3}-\frac{1}{6}n^{2/3}+1<-\frac{1}{7}n^{2/3}$$(where the last inequality holds for $n$ sufficiently large).  In addition, since $\{Y_n\}$ is equally likely to step towards or away from the root every time it is at a vertex in $\tilde{V}(T)^c$, it follows that $\underset{j\leq M(\delta n)}{\text{min}}\Phi(j)$  has the same distribution as the minimum value attained by simple random walk on $\mathbb{Z}$ up to time $\lceil{\delta n\rceil}-1$.  Hence, this means that the summand on the second line of \eqref{s2ndsi2s} (for large enough $n$ and all $i\leq \delta n$) is bounded above by $$2\cdot{\sf SRW}_{\mathbb{Z}}\Big(X_{\lfloor{\delta n\rfloor}}>\frac{1}{7}n^{2/3}\Big)\leq 2e^{-\big(\frac{1}{7\sqrt{\delta}}-1\big)n^{1/3}}$$(where the last inequality follows from first noting that ${\bf E}[e^{\lambda X_j}]\leq e^{\lambda ^2 j}$, then setting $\lambda=\frac{n^{-1/3}}{\sqrt{\delta}}$, and then applying Markov's inequality).  Multiplying this bound by $\delta n$, combining it with the bound for the expression on the third line of \eqref{s2ndsi2s}, and taking $\delta\downarrow 0$, now establishes \eqref{3ineqgwas2}.

To establish \eqref{3ineqgwas3}, we start by noting that if $T$ satisfies \eqref{lbrgwasem} for $n$ sufficiently large (recall this applies for ${\sf GW}_{\mathscr{A}}-$ a.s. every $T$), then \begin{align*}&\left\{\underset{j<i}{\text{max}}|Y_j|\leq \delta n^{1/3}\right\}\cap\left\{W(i)\leq n^{2/3}\right\}\cap\left\{S(i)\geq \delta^{\frac{1}{6}}n\right\}\\&\subseteq\left\{r(Y_j)\leq M\delta n^{1/3}\ \forall\ j<i\right\}\cap\left\{W(i)\leq n^{2/3}\right\}\cap\left\{S(i)\geq \delta^{\frac{1}{6}}n\right\}\\&\subseteq\left\{\underset{j<i}{\sum}r(Y_j)1_{Y_j\in\tilde{V}(T)}\leq M\delta n\right\}\cap\left\{S(i)\geq \delta^{\frac{1}{6}}n\right\}\\&\subseteq\left\{\underset{j<i}{\sum}r(Y_j)1_{Y_j\not\in\tilde{V}(T)}\geq(\delta^{\frac{1}{6}}-M\delta)n\right\}\subseteq\left\{\underset{j<\lfloor{\delta n\rfloor}}{\sum}r(Y_j)1_{Y_j\not\in\tilde{V}(T)}\geq(\delta^{\frac{1}{6}}-M\delta)n\right\}\end{align*}for $n$ sufficiently large and all $i\leq \delta n$.  Since the first event above is the event inside the summand in \eqref{3ineqgwas3}, it will suffice to establish that, for $\delta$ sufficiently small, we have \begin{equation}\label{bndsrvd1}{\sf SRW}_{{\bf T}}\Bigg(\underset{j<\lfloor{\delta n\rfloor}}{\sum}r(Y_j)1_{Y_j\not\in\tilde{V}({\bf T})}\geq(\delta^{\frac{1}{6}}-M\delta)n\Bigg)=o\Big(e^{-qn^{1/3}}\Big)\ \ {\sf GW}_{\mathscr{A}}-\text{a.s.}\end{equation}Now we define ${\sf GW}^i$ to be the Galton-Watson measure associated with the offspring distribution $(Z^i|Z^i\geq 1)$.  In other words, ${\sf GW}^i$ is the measure associated with ${\bf T}^i$ after conditioning on survival.  If we now let ${\sf GW}^f_j$ (for $j\geq 1$) represent the measure associated with ${\bf T}^f({\bf 0})$ after conditioning on the event $Z^i=j$, and then define ${\sf GW}^f_T:=\underset{v\in V(T)}{\prod}{\sf GW}^f_{Z^i(v)}$ (where $T$ can be any infinite rooted tree without leaves), then we find that ${\sf GW}_{\mathscr{A}}={\sf GW}^i\times{\sf GW}^f_{{\bf T}}$ (where the ${\bf T}$ in the subscript represents the tree generated according to the measure ${\sf GW}^i$).

Our approach for establishing \eqref{bndsrvd1} will be to use the above decomposition of ${\sf GW}_{\mathscr{A}}$ to show that for every infinite rooted tree $T$ without leaves (excluding the infinite half-line), the equality in \eqref{bndsrvd1} holds for ${\sf GW}^f_T-\text{a.s.}$ every ${\bf T}$ obtained by adding finite subtrees to the vertices of $T$.  To do this, we first observe that on account of the fact that for any fixed set of finite subtrees rooted at the ancestors of $v_o$ (with $v_o$ defined as in Lemma \ref{lemma:bndnravutn}), their contribution to the sum in \eqref{bndsrvd1} will become negligible relative to $\delta^{\frac{1}{6}}n$ as $\delta\downarrow 0$, this means that in order to show that \eqref{bndsrvd1} holds ${\sf GW}^f_T-\text{a.s.}$, it will suffice to show that, for $\delta$ sufficiently small, we have\begin{equation}\label{bndsrvd12}{\sf SRW}_{{\bf T}}\Bigg(\underset{j<\lfloor{\delta n\rfloor}}{\sum}r(Y_j)1_{Y_j\in\{\tilde{V}(T)^c\}\cap\{T(v_o)\}}\geq\frac{\delta^{\frac{1}{6}}}{2}n\Bigg)=o\Big(e^{-qn^{1/3}}\Big)\ \ {\sf GW}^f_T-\text{a.s.}\end{equation}Next we note that it follows from Lemma \ref{lemma:bndnravutn} that for $\delta$ sufficiently small, we have \begin{align}\label{bndmrvodc}{\sf SRW}_T\Big(\underset{v>v_o}{\text{max}}\big|\{j\leq \delta n:X_j=v\}\big|>\delta^{\frac{1}{4}}n^{2/3}\Big)&={\sf SRW}_T\Big(\underset{v>v_o}{\text{max}}\big|\{j\leq \delta n:X_j=v\}\big|>\delta^{-\frac{5}{12}}(\delta n)^{2/3}\Big)\\&=o\Big(e^{-\frac{\delta^{-\frac{1}{12}}}{5}n^{1/3}}\Big).\nonumber\end{align}Now we define $D$ to be the set of all paths of length $\lfloor{\delta n\rfloor}-1$ in $T$ that do not land on any vertex in $\{\tilde{V}(T)^c\}\cap\{T(v_o)\}$ more than $\delta^{\frac{1}{4}}n^{2/3}$ times.  Letting ${\bf y}_{[0,\lfloor{\delta n\rfloor}-1]}$ be one such path, if we let $a_1,\dots,a_{\ell}$ represent the number of times it hits each vertex in $\{\tilde{V}(T)^c\}\cap\{T(v_o)\}$ (only counting vertices it hits at least once), then we find that \begin{equation}\label{gwcSRWbndsr}{\sf GW}^f_T\times{\sf SRW}_{{\bf T}}\Bigg(\underset{j<\lfloor{\delta n\rfloor}}{\sum}r(Y_j)1_{Y_j\in\{\tilde{V}(T)^c\}\cap\{T(v_o)\}}\geq\frac{\delta^{\frac{1}{6}}}{2}n\ \bigg|\ {\bf y}_{[0,\lfloor{\delta n\rfloor}-1]}\Bigg)={\bf P}\Bigg(\sum_{j\leq\ell}a_j L_j\geq\frac{\delta^{\frac{1}{6}}}{2}n\Bigg)\end{equation}(where the $L_j$'s represent independent copies of $r({\bf 0})|_{Z^i=1}$).  Since the $L_j$'s have exponential moments, this means there must exist $\lambda_o>0$ and $c<\infty$ such that for any nonnegative $\lambda\leq\lambda_o$ we have ${\bf E}[e^{\lambda L_1}]\leq e^{c\lambda}$.  If we now let $b_j=\delta^{-\frac{1}{4}}n^{-2/3}\lambda_o a_j$ for each $j$, then since $a_j\leq \delta^{\frac{1}{4}}n^{2/3}$ for each $j$ and $a_1+\dots a_{\ell}\leq \delta n$, it follows that $${\bf E}[e^{b_1 L_1+\dots +b_{\ell}L_{\ell}}]=\prod_{i=1}^{\ell}{\bf E}[e^{b_i L_i}]\leq e^{c(b_1+\dots+b_{\ell})}\leq e^{c\lambda_o \delta^{\frac{3}{4}}n^{1/3}},$$which implies, by Markov's inequality, that $${\bf P}\Big(a_1 L_1+\dots +a_{\ell}L_{\ell}\geq\frac{\delta^{\frac{1}{6}}}{2}n\Big)={\bf P}\Big(b_1 L_1+\dots +b_{\ell}L_{\ell}\geq\frac{\delta^{-\frac{1}{12}}}{2}\lambda_o n^{1/3}\Big)\leq e^{\Big(c \delta^{\frac{3}{4}}-\frac{\delta^{-\frac{1}{12}}}{2}\Big)\lambda_o n^{1/3}}.$$Plugging this bound into the expression on the left in \eqref{gwcSRWbndsr} and summing over all ${\bf y}_{[0,\lfloor{\delta n\rfloor}-1]}\in D$, while also noting that \eqref{bndmrvodc} implies that ${\sf SRW}_T(D^c)=o\Big(e^{-\frac{\delta^{-\frac{1}{12}}}{5}n^{1/3}}\Big)$ for $\delta$ sufficiently small, we now find that \begin{align*}&{\sf GW}^f_T\times{\sf SRW}_{{\bf T}}\Bigg(\underset{j<\lfloor{\delta n\rfloor}}{\sum}r(Y_j)1_{Y_j\in\{\tilde{V}(T)^c\}\cap\{T(v_o)\}}\geq\frac{\delta^{\frac{1}{6}}}{2}n\Bigg)\leq{\sf SRW}_T(D^c) \\&+\sum_D{\sf SRW}_T\Big({\bf y}_{[0,\lfloor{\delta n\rfloor}-1]}\Big)\cdot{\sf GW}^f_T\times{\sf SRW}_{{\bf T}}\Bigg(\underset{j<\lfloor{\delta n\rfloor}}{\sum}r(Y_j)1_{Y_j\in\{\tilde{V}(T)^c\}\cap\{T(v_o)\}}\geq\frac{\delta^{\frac{1}{6}}}{2}n\ \bigg|\ {\bf y}_{[0,\lfloor{\delta n\rfloor}-1]}\Bigg)\\ &= o\Big(e^{-r_o \delta^{-\frac{1}{12}} n^{1/3}}\Big)\end{align*}for some $r_o>0$ for small enough $\delta$.  By another application of Markov's inequality, this then implies that $${\sf GW}^f_T\Bigg({\bf T}:{\sf SRW}_{{\bf T}}\Bigg(\underset{j<\lfloor{\delta n\rfloor}}{\sum}r(Y_j)1_{Y_j\in\{\tilde{V}(T)^c\}\cap\{T(v_o)\}}\geq\frac{\delta^{\frac{1}{6}}}{2}n\Bigg)\geq n^2 e^{-r_o \delta^{-\frac{1}{12}} n^{1/3}}\Bigg)=o\Big(\frac{1}{n^2}\Big)$$for $\delta$ sufficiently small, which by Borel-Cantelli, implies that \begin{equation}\label{bndSRWtgwftbc}{\sf SRW}_{{\bf T}}\Bigg(\underset{j<\lfloor{\delta n\rfloor}}{\sum}r(Y_j)1_{Y_j\in\{\tilde{V}(T)^c\}\cap\{T(v_o)\}}\geq\frac{\delta^{\frac{1}{6}}}{2}n\Bigg)=O\Big(n^2 e^{-r_o \delta^{-\frac{1}{12}} n^{1/3}}\Big)\ \ {\sf GW}^f_T-\text{a.s.}\end{equation}Taking $\delta\downarrow 0$ in \eqref{bndSRWtgwftbc}, this now establishes \eqref{bndsrvd12}, thus completing the proof of the proposition.
\end{proof}

\medskip
\begin{proof}[Proof of Theorem \ref{brbnd0c}]
We start the proof of the upper bound in \eqref{ubnd0c} by introducing a few additional definitions.  First, for any infinite rooted tree $T$ and vertex $v\in T$, let $\mathcal{D}_T(v)$ equal the graph distance between $v$ and the closest vertex to $v$ that is in $T^i$.  Next, for every $n\geq 0$ we define $\mathcal{H}_n(T)$ to be the maximum height of any finite subtree rooted at a vertex on level $n$ of $T$, and we note that, assuming $\mathcal{H}_n(T)\geq 1$, it can be expressed as $1+\text{max}\{h(v):v\in \partial T_{n+1},\ \mathcal{D}_T(v)=1\}$ (where $h(v)$ represents the height of the tree $T(v)$).  Now observe that for any $A>0$, $r>\mu$, and $n\geq 1$ such that $An>1$, we have \begin{align}{\sf GW}_{\mathscr{A}}\big(\mathcal{H}_n({\bf T})\geq An)&\leq {\sf GW}_{\mathscr{A}}\big(\big|\{v\in \partial{\bf T}_{n+1}:\mathcal{D}_{{\bf T}}(v)=1\}\big|\geq r^{n+1}\big)+1-\Big(1-{\sf GW}\big(h({\bf 0})\geq An-1\ \big|\ \mathscr{A}^c\big)\Big)^{r^{n+1}} \nonumber \\ &\leq{\sf GW}_{\mathscr{A}}\big(Z_{n+1}\geq r^{n+1}\big)+1-\big(1-\mu_o^{An-1}\big)^{r^{n+1}} \nonumber \\ &\leq\frac{(\mu/r)^{n+1}}{1-q}+1-\big(1-\mu_o^{An-1}\big)^{r^{n+1}}\label{ubplftlngcn}\end{align}(recall that $\mu_o={\bf E}[Z|\mathscr{A}^c]$).  Since the last expression in \eqref{ubplftlngcn} is summable for $A>\frac{\text{log}r}{\text{log}(1/\mu_o)}$, it follows from Borel-Cantelli that for any $A>\frac{\text{log}\mu}{\text{log}(1/\mu_o)}$, we have ${\sf GW}_{\mathscr{A}}(\mathcal{H}_n({\bf T})\geq An\ \text{i.o.})=0$.  Hence, we can conclude that in order to establish the upper bound in \eqref{ubnd0c}, it will suffice to show that \begin{equation}\label{ncptt}\underset{A\to\infty}{\text{lim}}\Bigg[\underset{n\to\infty}{\text{liminf}}\ {\sf SRW}_{\bf{T}}\bigg(\text{max}\{|X_j|:j\leq 2n,\ X_j\in{\bf T}^i\}\leq An^{1/3}\ \Big|\ X_{2n}={\bf 0}\bigg)\Bigg]=1\ \ {\sf GW}_{\mathscr{A}}-\text{a.s.}\end{equation}

To establish \eqref{ncptt}, we'll use nearly the same argument used to get \eqref{impbndfr16} and \eqref{limforpreq10} at the end of the proof of Theorem \ref{theorem:diffbndp1p}.  So let $T$ be a tree in $\mathscr{A}_k$ for some $k$ (see Lemma \ref{lemma:mdmc0c} for the definition of $\mathscr{A}_k$) that satisfies the equality in Proposition \ref{prop:lbfrp2}, and assume $T^i$ satisfies Proposition \ref{prop:lbonpolt} for some $p>0$, $s>0$, and $N\in\mathbb{N}$.  Then for any $v\in T^i$ and $A$, $n$ satisfying $|v|\geq A n^{1/3}\geq N$, the probability simple random walk beginning at $v$ ever returns to the root is bounded above by $e^{s\text{log}(1-p)An^{1/3}}$, thus implying that for $An^{1/3}\geq N$ we have $${\sf SRW}_T\bigg(\Big\{\text{max}\{|X_j|:j\leq 2n,\ X_j\in T^i\}\geq A n^{1/3}\Big\}\cap\{X_{2n}={\bf 0}\}\bigg)\leq 2n e^{s\text{log}(1-p)A n^{1/3}}.$$Hence, if $A>\frac{C_k}{s\text{log}(1-p)}$ (see Proposition \ref{prop:lbfrp2}), then we find that \begin{equation}\label{eqfnmxj}\underset{n\to\infty}{\text{liminf}}\ {\sf SRW}_{T}\bigg(\text{max}\{|X_j|:j\leq 2n,\ X_j\in T^i\}\leq An^{1/3}\ \Big|\ X_{2n}={\bf 0}\bigg)=1.\end{equation}Since ${\bf T}^i$ is itself a Galton-Watson tree (with respect to ${\sf GW}_{\mathscr{A}}$), this means it must satisfy Proposition \ref{prop:lbonpolt} ${\sf GW}_{\mathscr{A}}-$almost surely.  In addition, because ${\sf GW}_{\mathscr{A}}(\mathscr{A}_k)\to 1$ as $k\to\infty$ (this follows from the fact that the offspring distribution $Z\cdot 1_{Z\leq k}$ converges to $Z$), we also know that ${\sf GW}_{\mathscr{A}}-$almost surely every $T$ satisfies Proposition \ref{prop:lbfrp2}.  Hence, the properties we imposed on $T$ to get \eqref{eqfnmxj} apply to ${\sf GW}_{\mathscr{A}}-$almost surely every tree $T$, thus establishing \eqref{ncptt} and completing the proof of the upper bound in \eqref{ubnd0c}.

On account of Proposition \ref{prop:lbfrp2}, in order to obtain the lower bound in \eqref{ubnd0c}, it will suffice to show that ${\sf GW}_{\mathscr{A}} -$a.s. every $T$ has the property that for each $q<\infty$, taking $\delta>0$ sufficiently small gives\begin{equation}\label{pSRWsunylo}{\sf SRW}_T\Big(\underset{j\leq 2n}{\text{max}}|X_j|\leq \delta n^{1/3}\Big)=o\Big(e^{-qn^{1/3}}\Big).\end{equation}Now to establish \eqref{pSRWsunylo}, we will need to refer back to the sequences $\{N_j\}$ and $\{Y_n\}$, as well as the functions $S(i)$ and $W(i)$, from Propostion \ref{prop:3ineqforlb}.  We begin by defining the quantity $t(n):=\text{min}\{j:N_j>2n\}$, and note that because $$\underset{j\leq 2n}{\text{max}}|X_j|\geq\underset{j<t(n)}{\text{max}}|Y_j|,$$proving \eqref{pSRWsunylo} is reduced to showing that ${\sf GW}_{\mathscr{A}}-$a.s. every $T$ satisfies the condition that for each $q<\infty$,\begin{equation}\label{bpSRWlnyutn}{\sf SRW}_T\Big(\underset{j<t(n)}{\text{max}}|Y_j|\leq \delta n^{1/3}\Big)=o\Big(e^{-qn^{1/3}}\Big)\end{equation}for $\delta$ sufficiently small.  To establish \eqref{bpSRWlnyutn}, we let $T$ be any surviving tree and observe that for any $\delta>0$, we have\begin{align}{\sf SRW}_T\Big(\underset{j<t(n)}{\text{max}}|Y_j|\leq \delta n^{1/3}\Big)&={\sf SRW}_T\Big(\underset{j<t(n)}{\text{max}}|Y_j|\leq \delta n^{1/3},\ t(n)>\delta n\Big)+{\sf SRW}_T\Big(\underset{j<t(n)}{\text{max}}|Y_j|\leq \delta n^{1/3},\ t(n)\leq \delta n\Big) \nonumber\\ &\leq {\sf SRW}_{ T^i}\Big(\underset{j\leq \delta n}{\text{max}}|X_j|\leq \delta n^{1/3}\Big)+{\sf SRW}_T\Big(\underset{j<t(n)}{\text{max}}|Y_j|\leq \delta n^{1/3},\ t(n)\leq \delta n\Big)\label{lsineqBRWplny}\end{align}Now coupling simple random walk on $T^i$ with simple random walk on $\mathbb{Z}$ via the same method as in the first paragraph of the proof of Theorem \ref{theorem:diffbndp1p}, and once again using Lemma \ref{lemma:t3mogl34g}, we find that $${\sf SRW}_{T^i}\Big(\underset{j\leq \delta n}{\text{max}}|X_j|\leq \delta n^{1/3}\Big)\leq e^{-(1+o(1))\frac{\pi^2}{8\delta}n^{1/3}},$$which is $o(e^{-q n^{1/3}}\Big)$ for small enough $\delta$.  Combining this with \eqref{bpSRWlnyutn} and \eqref{lsineqBRWplny}, we conclude that to establish \eqref{pSRWsunylo}, it will suffice to show that ${\sf GW}_{\mathscr{A}}-$a.s. every $T$ satisfies the condition that for every $q<\infty$, taking $\delta$ sufficiently small gives\begin{equation}\label{bpanytnlna}{\sf SRW}_T\Big(\underset{j<t(n)}{\text{max}}|Y_j|\leq \delta n^{1/3},\ t(n)\leq \delta n\Big)=o\Big(e^{-qn^{1/3}}\Big).\end{equation}Now noting that\begin{align}\label{ubmyjtnlnas}{\sf SRW}_T\Big(\underset{j<t(n)}{\text{max}}|Y_j|\leq \delta n^{1/3},\ t(n)\leq \delta n\Big)&\leq\sum_{i=1}^{\lfloor{\delta n\rfloor}}{\sf SRW}_T\Big(\underset{j<i}{\text{max}}|Y_j|\leq \delta n^{1/3},\ t(n)\leq i\Big) \\ &=\sum_{i=1}^{\lfloor{\delta n\rfloor}}{\sf SRW}_T\Big(\underset{j<i}{\text{max}}|Y_j|\leq \delta n^{1/3},\ N_i>2n\Big)\nonumber \\&\leq \sum_{i=1}^{\lfloor{\delta n\rfloor}}{\sf SRW}_T\Big(\underset{j<i}{\text{max}}|Y_j|\leq \delta n^{1/3},\ N_i>2n,\ S(i)<\delta^{\frac{1}{6}}n\Big)\nonumber \\&+\sum_{i=1}^{\lfloor{\delta n\rfloor}}{\sf SRW}_T\bigg(\underset{j<i}{\text{max}}|Y_j|\leq \delta n^{1/3},\ W(i)> n^{2/3}\bigg)\nonumber \\&+\sum_{i=1}^{\lfloor{\delta n\rfloor}}{\sf SRW}_T\bigg(\underset{j<i}{\text{max}}|Y_j|\leq \delta n^{1/3},\ W(i)\leq n^{2/3},\ S(i)\geq \delta^{\frac{1}{6}}n\bigg),\nonumber\end{align}and observing that the expressions on the last three lines of \eqref{ubmyjtnlnas} are simply the three sums from \eqref{3ineqgwas}, \eqref{3ineqgwas2}, and \eqref{3ineqgwas3} respectively, we see that \eqref{bpanytnlna} now follows from Proposition \ref{prop:3ineqforlb}, thus establishing the lower bound in \eqref{ubnd0c}, and completing the proof of the theorem.
\end{proof}

\section{Case 2: ${\bf P}(Z\geq 2)=1$}

\begin{theorem}\label{theorem:diffbound}
If the offspring distribution $Z$ satisfies ${\bf P}(Z\geq 2)=1$, and ${\bf E}[Z]<\infty$, then for every $\gamma<1$
\begin{equation}\label{bndonmaxR}\underset{n\to\infty}{\emph{lim}}{\sf SRW}_{\bf{T}}\bigg(\underset{j\leq 2n}{\emph{max}}|X_j|\geq n^{\gamma}\ \Big|\ X_{2n}={\bf 0}\bigg)=1\ \ {\sf GW}-\emph{a.s.},\end{equation}and for every $\beta<1$\begin{equation}\label{ubndonmaxd}\underset{n\to\infty}{\emph{lim}}{\sf SRW}_{\bf{T}}\bigg(\underset{j\leq 2n}{\emph{max}}|X_j|\geq\frac{n}{(\emph{log}\ n)^{\beta}}\ \Big|\ X_{2n}={\bf 0}\bigg)=0\ \ {\sf GW}-\emph{a.s.}\end{equation}
\end{theorem}

\bigskip
\noindent
In order to prove the above result, we will need to begin by establishing two lemmas, the first of which involves comparing ${\sf SRW}_T$ to a different measure associated with a specific type of biased random walk on $T$ that we represent as $\{X_j\}^{\beta}$.  To define this measure, which will be denoted as ${\sf BRW}_T$, we let $m:=\text{min}\{j:{\bf P}(Z=j)>0\}$, and we define the transition probabilities for our biased random walk as follows: 

\medskip
\noindent
(i) Let $v^{(i)}$ represent any child of a vertex $v$.  If $v$ is the root then set $${\sf BRW}_T(X_{j+1}=v^{(i)}|X_j=v)=\frac{1}{\text{deg}(v)}.$$

\medskip
\noindent
(ii) If $v$ is a non-root vertex with parent $u$, then set $${\sf BRW}_T(X_{j+1}=u|X_j=v)=\frac{m}{\text{deg}(v)+m},\ \text{and}\ {\sf BRW}_T(X_{j+1}=v^{(i)}|X_j=v)=\frac{1}{\text{deg}(v)+m}$$ where $v^{(i)}$ again represents any child of $v$.

\bigskip
\noindent
Having defined ${\sf BRW}_T$, we can now state the first lemma.  As noted in the introduction, this lemma, as well as its application to the proof of \eqref{bndonmaxR}, is largely adapted from a method used by Gantert and Peterson in \cite{GP} (see Lemma 4.2 and Proposition 4.3 in their paper).

\begin{lemma}\label{lemma:compSRWm}
Let $M=\frac{4m}{(m+1)^2}$ and define the random variable $$B_n:=\big|\{j<2n:X_j ={\bf 0}\ \text{or }\emph{deg}(X_j)>m\}\big|.$$ Then there exist constants $c_1,c_2\in (0,1)$ (depending only on the value of $m$) such that for any surviving tree $T$ for which all vertices have degree at least $m$, and any event $A$ in the path space of $T$ that depends only on the first $2n$ steps, and for which all paths in $A$ begin at ${\bf 0}$ and satisfy $X_{2n}={\bf 0}$, we have \begin{equation}\label{pbounds}M^n{\bf E}_{{\sf BRW}_T}\Big[c_1^{B_n}{\bf 1}_A\Big]\leq{\sf SRW}_T(A)\leq M^n{\bf E}_{{\sf BRW}_T}\Big[c_2^{B_n}{\bf 1}_A\Big].\end{equation}
\end{lemma}

\begin{proof}
Letting $X_{[0,2n]}=(X_0,X_1,\dots,X_{2n})$ represent the path of the random walk inside of our tree $T$ up to time $2n$, we get \begin{equation}\label{rndforA}{\sf SRW}_T(A)={\bf E}_{{\sf BRW}_T}\bigg[\frac{d{\sf SRW}_T}{d{\sf BRW}_T}X_{[0,2n]}{\bf 1}_A\bigg]\end{equation} where for any ${\bf x}_{[0,2n]}=(x_0,x_1,\dots,x_{2n})$ we have $$\frac{d{\sf SRW}_T}{d{\sf BRW}_T}({\bf x}_{[0,2n]})=\frac{{\sf SRW}_T(X_j=x_j\ \forall\ j\leq 2n)}{{\sf BRW}_T(X_j=x_j\ \forall\ j\leq 2n)}=\prod_{j=0}^{2n-1}\frac{{\sf SRW}_T(X_{j+1}=x_{j+1}|X_j=x_j)}{{\sf BRW}_T(X_{j+1}=x_{j+1}|X_j=x_j)}.$$Using the expressions for the transition probabilities of $\{X_j\}^{\beta}$ given above, we observe that for each $j$ $$\frac{{\sf SRW}_T(X_{j+1}=x_{j+1}|X_j=x_j)}{{\sf BRW}_T(X_{j+1}=x_{j+1}|X_j=x_j)}=\left\{\begin{array}{ll}\frac{\text{deg}(x_j)+m}{\text{deg}(x_j)+1}&\text{if }x_j\neq{\bf 0}\text{ and }|x_{j+1}|=|x_j|+1\\ \frac{\text{deg}(x_j)+m}{m(\text{deg}(x_j)+1)}&\text{if }x_j\neq{\bf 0}\text{ and }|x_{j+1}|=|x_j|-1 \\ 1 &\text{if }x_j={\bf 0}\end{array}\right.$$Since $X_{2n}=\bf{0}$ implies exactly $n$ out of the first $2n$ steps are towards the root, it now follows that \begin{equation}\label{simpforpr}\frac{d{\sf SRW}_T}{d{\sf BRW}_T}({\bf x}_{[0,2n]})=m^{-n}\prod_{j=0}^{2n-1}\bigg(\frac{\text{deg}(x_j)+m}{\text{deg}(x_j)+1}-\frac{m-1}{\text{deg}(x_j)+1}{\bf 1}_{\{x_j={\bf 0}\}}\bigg)\ \ \ \ \forall\ {\bf x}_{[0,2n]}\in\{X_{2n}={\bf 0}\}.\end{equation}Noting that the expression inside the product is maximized at $\text{deg}(x_j)=m$, where it takes the value $\frac{2m}{m+1}$, and observing that it is strictly decreasing with respect to $\text{deg}(x_j)$ for $x_j\neq\bf{0}$ (while the case where $x_j=\bf{0}$ gives the minimum value which is $1$), we see that there must exist $c_1,c_2\in(0,1)$ (depending only on $m$) such that $$c_1\frac{2m}{m+1}\leq\frac{\text{deg}(v)+m}{\text{deg}(v)+1}-\frac{m-1}{\text{deg}(v)+1}{\bf 1}_{\{v={\bf 0}\}}\leq c_2\frac{2m}{m+1} \ \ \ \ \forall\ v\ \text{s.t. }v={\bf 0}\ \text{or }\text{deg}(v)>m.$$Hence, combining this with \eqref{simpforpr} we obtain the following inequalities:\begin{equation}\label{bndsforlpr}m^{-n}\Big(\frac{2m}{m+1}\Big)^{2n}c_1^{B_n}\leq\frac{d{\sf SRW}_T}{d{\sf BRW}_T}({\bf x}_{[0,2n]})\leq m^{-n}\Big(\frac{2m}{m+1}\Big)^{2n}c_2^{B_n}\ \ \ \ \forall\ {\bf x}_{[0,2n]}\in\{X_{2n}={\bf 0}\}.\end{equation}Noting that $m^{-n}\Big(\frac{2m}{m+1}\Big)^{2n}=M^n$ and combining \eqref{bndsforlpr} with \eqref{rndforA} now completes the proof of the lemma.  
\end{proof}

\bigskip
In order to present the second lemma that will be needed in the proof of Theorem \ref{theorem:diffbound}, we'll need the following definitions: First, for any tree $T$ and any vertex $v\in T$, let $w_T(v)$ represent the maximum depth of an $m$-regular tree rooted at $v$, i.e. $w_T(v)$ is the maximum value for which all $j$th generation descendants of $v$ (for $j<w_T(v)$) have exactly $m$ offspring.  In addition, for every $n\geq 0$, we define $W_{n,k}(T):=\text{max}\{d_T(v):|v|=n,\ \text{deg}(u)\leq k\ \forall\ u<v\}$ (note that when $m=1$, $w_T(v)$ and $W_{n,k}(T)$ are equivalent to the quantities $d_T(v)$ and $D_{n,k}(T)$ that were defined above Lemma \ref{lemma:maxdepthline}).  Unless otherwise stated, additional notation appearing below that was also in the statement and proof of Lemma \ref{lemma:maxdepthline}, will be assumed to have the same definitions as earlier.

\medskip
\begin{lemma}\label{lemma:maxdepthreg}
If $k>m$ satisfies \begin{equation}\label{condtrunc}\sum_{j=m}^k {\bf P}(Z=j)\cdot j>1,\end{equation}then \begin{equation}\label{refordnk}D_{n,k}({\bf T})=(1+o(1))\frac{\emph{log}\ n}{\emph{log}\ m}\ \ {\sf GW}-\emph{a.s}.\end{equation} conditioned on $\mathscr{A}_k$.
\end{lemma}

\begin{proof}
For any vertex $v\in\tilde{T}^{(k)}$, let $d^*_T(v)$ equal the height of the largest m-ary subtree that is in both $T$ and $\tilde{T}^{(k)}$, and is rooted at $v$.  In addition, let $D^*_{n,k}=\text{max}\{d^*_T(v):|v|=n,\ v\in \tilde{T}^{(k)}\}$, and note that $D^*_{n,k}\leq D_{n,k}$.  Now letting $\rho_k$ refer to the value ${\bf P}(Z=m)\cdot\Big({\sf GW}(\mathscr{A}_k)\Big)^{m-1}$, we find that for any $\ell>0$, \begin{equation}\label{bndingpl}{\sf GW}^{(k)}\big(D^*_{n,k}<\ell\ |\ Z^{(k)}_n\big)=\Big(1-\rho_k^{1+m+\dots +m^{\ell -1}}\Big)^{Z^{(k)}_n}=\Big(1-\rho_k^{\frac{m^{\ell}-1}{m-1}}\Big)^{Z^{(k)}_n}.\end{equation}Now if we have $r,\delta$ where $r>1$ and $0<\delta<1$, then it follows from \eqref{bndingpl} that \begin{equation}\label{bndcondgw}{\sf GW}^{(k)}\Big(D^*_{n,k}<(1-\delta)\frac{\text{log}\ n}{\text{log}\ m}\ \Big|\ Z^{(k)}_n\geq r^n\Big)\leq\Big(1-\rho_k^{\frac{m}{m-1}n^{1-\delta}}\Big)^{r^n}\leq e^{-r^n\rho_k^{\frac{m}{m-1}n^{1-\delta}}}=e^{-r^{n+o(n)}}.\end{equation} If we then define the sequence $\{v_n\}$ in $\tilde{{\bf T}}_k$ (conditioned to survive) so that $v_0={\bf 0}$ and, for every $j\geq 1$, $v_j$ is the left most child of $v_{j-1}$, then by a simple application of large deviations, there must exist constants $c,r'>0$ such that the number of $v_j$ (for $j\leq \frac{n}{2}$) that have at least one additional child besides $v_{j+1}$, is greater than $cn$ with probability at least $1-e^{-r'n}$. Since the subtrees rooted at each of these additional children are independent, and because $Z_n^{(k)}$ is greater than or equal to $\tilde{\mu}_k^n$ with probability bounded away from $0$ (this follows from a basic martingale argument), we can conclude that  there exist $r_1,r_2>0$ such that $${\sf GW}^{(k)}\big(Z^{(k)}_n\leq (1+r_1)^n\big)\leq e^{-r_2 n}.$$Combining this last inequality with \eqref{bndcondgw}, and recalling that $D^*_{n,k}\leq D_{n,k}$, we now observe that $${\sf GW}^{(k)}\Big(D_{n,k}<(1-\delta)\frac{\text{log}\ n}{\text{log}\ m}\Big)\leq {\sf GW}^{(k)}\Big(D^*_{n,k}<(1-\delta)\frac{\text{log}\ n}{\text{log}\ m}\Big)\leq e^{-(1+r_1)^{n+o(n)}}+e^{-r_2 n}.$$Since the expression on the right is summable and $\delta$ was arbitrary, it now follows from Borel-Cantelli that \begin{equation}\label{linfbnd}\text{liminf}\frac{D_{n,k}}{(\text{log}\ n/\text{log}\ m)}\geq 1\ \ \ \ {\sf GW}^{(k)}-\text{a.s.}\end{equation}

For the other direction we note that, since $D_{n,k}$ is increasing with respect to $k$, it suffices to just prove the desired inequality for the case of $k=\infty$ (note ${\sf GW}^{(\infty)}$ and ${\sf GW}$ are the same, as are $D^*_{n,\infty}$ and $D_{n,\infty}$).  Hence, if we take any $L>\mu$ (recall $\mu={\bf E}[Z]$) and let $\alpha:={\bf P}(Z=m)$, then it follows from \eqref{bndingpl} that\begin{align}\label{bndotherdir}{\sf GW}\Big(D_{n,\infty}\geq (1+\delta)\frac{\text{log}\ n}{\text{log}\ m}\Big)&\leq 1-\Big(1-\alpha^{\frac{n^{1+\delta}-1}{m-1}}\Big)^{L^n}+{\sf GW}(Z^{(\infty)}_n>L^n) \\ &\leq 1-e^{-(1+o(1))\alpha^{\frac{n^{1+\delta}-1}{m-1}}L^n}+(\mu/L)^n=\alpha^{n^{1+\delta+o(1)}}+(\mu/L)^n.\nonumber\end{align}Since this is once again summable and $\delta$ is arbitrary, we now find that for every $k$ satisfying \eqref{condtrunc}, $$\text{limsup}\frac{D_{n,k}}{(\text{log}\ n/\text{log}\ m)}\leq 1\ \ \ \ {\sf GW}^{(k)}-\text{a.s.}$$Combining this with \eqref{linfbnd} now completes the proof of the lemma.
\end{proof}

\medskip
\begin{proof}[Proof of Theorem \ref{theorem:diffbound}]
In order to establish \eqref{bndonmaxR}, it will suffice to show that for every $\gamma_1,\gamma_2$ satisfying $0<\gamma_1<\gamma_2<1$, we have $$\underset{n\to\infty}{\text{lim}}\frac{{\sf SRW}_{\bf{T}}\bigg(\underset{j\leq 2n}{\text{max}}|X_j|\leq n^{\gamma_2},\ X_{2n}={\bf 0}\bigg)}{{\sf SRW}_{\bf{T}}\bigg(\underset{j\leq 2n}{\text{max}}|X_j|\leq n^{\gamma_1},\  X_{2n}={\bf 0}\bigg)}=\infty\ \ {\sf GW}\text{--a.s.}$$We will accomplish this by proving an even stronger statement, which is that for any $\gamma\in (0,1)$, \begin{equation}\label{preqforpr}\underset{n\to\infty}{\text{lim}}\frac{(\text{log}\ n)^2}{n}\bigg[\text{log}\Big[{\sf SRW}_{{\bf T}}\big(\underset{j\leq 2n}{\text{max}}|X_j|\leq n^{\gamma},\ X_{2n}={\bf 0}\big)\Big]-n\text{log}\ M\bigg]=\frac{-(\pi\text{log}\ m)^2}{\gamma^2}\ \ \ \ {\sf GW}\text{--a.s.}\end{equation}Our first step in proving this last equality is to note that\begin{align}\label{inflrbnd}{\sf SRW}_{\bf{T}}\bigg(\underset{j\leq 2n}{\text{max}}|X_j|\leq n^{\gamma},\ X_{2n}={\bf 0}\bigg)&\geq{\sf SRW}_{\bf{T}}\bigg(\underset{j\leq 2n}{\text{max}}|X_j|\leq n^{\gamma},\ X_{2n}={\bf 0},\ B_n\leq\frac{2n}{(\text{log}\ n)^3}\bigg) \\ &\geq M^n c_1^{2n/(\text{log}\ n)^3}{\sf BRW}_{\bf{T}}\bigg(\underset{j\leq 2n}{\text{max}}|X_j|\leq n^{\gamma},\ X_{2n}={\bf 0},\ B_n\leq\frac{2n}{(\text{log}\ n)^3}\bigg)\ \ {\sf GW}-\text{a.s.}\nonumber\end{align}(where the second inequality follows from Lemma \ref{lemma:compSRWm}).  Now observe that on account of Lemma \ref{lemma:maxdepthreg} (along with \eqref{embtbnddeg} from section 2), we know that for ${\sf GW}-\text{a.s.}$ every $T$, there exists an integer $L\geq m$ (for which \eqref{condtrunc} is satisfied) such that $T$ satisfies the equality in \eqref{refordnk} for every $k\geq L$ (where the value of $L$ can depend on the specific tree).  Hence, in order to prove \eqref{preqforpr}, it will suffice to show that the equality in it holds for every such $T$, provided we also assume $T$ has minimum degree no smaller than $m$.  

Based on the properties of $T$, we know that there must exist a sequence of vertices $v_n\in T$ such that $|v_n|=(1+o(1))n^{\gamma}$, $\text{deg}(u)\leq L\ \forall\ u<v_n$, and where for each $n$ there exists an $m$-ary subtree of height $h_n=(1+o(1))\frac{\gamma\text{log}\ n}{\text{log}\ m}$ rooted at $v_n$, such that $|v_n|+h_n\leq n^{\gamma}$.  Next, we define the event $E_n$ as the intersection of the following three events: First, we let $E_{n,1}$ be the event that the random walk on $T$ arrives at $v_n$ at time $|v_n|$, and then takes another $\lfloor{h_n/2\rfloor}$ consecutive steps away from the root.  $E_{n,2}$ we define to be the event that the walk is inside the $m$-ary subtree rooted at $v$ at time $|v_n|+\lfloor{h_n/2\rfloor}$, and stays inside of it until a time $2n-k$, at which point it resides at a vertex at height $k$.  Finally, $E_{n,3}$ is defined to be the event that the walk is at the root at time $2n$.  Now we observe that because we're assuming that $\gamma<1$, it follows that for all $n$ sufficiently large, we have \begin{equation}\label{econtint}E_n\subseteq\{\underset{j\leq 2n}{\text{max}}|X_j|\leq n^{\gamma}\}\cap\{X_{2n}={\bf 0}\}\cap\{B_n\leq 2n/(\text{log}\ n)^3\}.\end{equation}In addition, we note that \begin{equation}\label{lbndforen}{\sf BRW}_T(E_n)=\prod_{j=1}^3{\sf BRW}_T(E_{n,j}|\{E_{n,1}\cup\dots\cup E_{n,j-1}\})\geq(L+m)^{-2n^{\gamma}}{\sf BRW}_T(E_{n,2}|E_{n,1}),\end{equation}where the inequality on the right follows from the definitions of $\sf{BRW}_T$ and the $E_{n,j}$'s, and the fact that all of the ancestors of $v_n$ have degree less than or equal to $L$.  Since the definition of ${\sf BRW}_T$ also implies that for any vertex $v$ inside the $m$-regular subtree rooted at $v_n$ we have $${\sf BRW}_T(|X_{j+1}|=|v|+1\ |\ X_j=v)=\frac{1}{2},$$ we can also conclude that \begin{equation}\label{pforSRWonzbnd}{\sf BRW}_T(E_{n,2}|E_{n,1})\geq{\sf SRW}_{\mathbb{Z}}\big(\underset{j\leq 2n}{\text{max}}|X_j|\leq\Bigl\lfloor{\frac{h_n-1}{2}\Bigr\rfloor}\big).\end{equation}Now combining \eqref{pforSRWonzbnd} with Lemma \ref{lemma:t3mogl34g}, and recalling that $h_n=(1+o(1))\frac{\gamma\text{log}\ n}{\text{log}\ m}$, we see that \begin{align*}\frac{h_n^2}{n}\text{log}\Big[{\sf SRW}_{\mathbb{Z}}\Big(\underset{j\leq 2n}{\text{max}}|X_j|\leq\Bigl\lfloor{\frac{h_n-1}{2}\Bigr\rfloor}\Big)\Big]=-\pi^2+o(1)&\implies\frac{\Big(\frac{\gamma\text{log}\ n}{\text{log}\ m}\Big)^2}{n}\text{log}\Big[{\sf SRW}_{\mathbb{Z}}\Big(\underset{j\leq 2n}{\text{max}}|X_j|\leq\Bigl\lfloor{\frac{h_n-1}{2}\Bigr\rfloor}\Big)\Big]=-\pi^2+o(1) \\ &\implies\frac{\Big(\frac{\gamma\text{log}\ n}{\text{log}\ m}\Big)^2}{n}\text{log}\Big[{\sf BRW}_T(E_{n,2}|E_{n,1})\Big]\geq -\pi^2+o(1).\end{align*}Alongside \eqref{econtint} and \eqref{lbndforen}, this now implies that \begin{align*}\frac{\Big(\frac{\gamma\text{log}\ n}{\text{log}\ m}\Big)^2}{n}\text{log}&\Big[{\sf BRW}_T\Big(\underset{j\leq 2n}{\text{max}}|X_j|\leq n^{\gamma}, X_{2n}={\bf 0},B_n\leq 2n/(\text{log}\ n)^3\Big)\Big] \\ &\geq\frac{\Big(\frac{\gamma\text{log}\ n}{\text{log}\ m}\Big)^2}{n}\bigg(-2n^{\gamma}\text{log}(L+m)+\text{log}\big[{\sf BRW}_T(E_{n,2}|E_{n,1})\big]\bigg)\geq -\pi^2+o(1).\end{align*}Finally, combining this with \eqref{inflrbnd} we get \begin{equation}\label{feqlbndlog}\frac{(\text{log}\ n)^2}{n}\bigg[\text{log}\Big[{\sf SRW}_T\big(\underset{j\leq 2n}{\text{max}}|X_j|\leq n^{\gamma},\ X_{2n}={\bf 0}\big)\Big]-n\text{log}M\bigg]\geq -\Big(\frac{\pi\text{log}\ m}{\gamma}\Big)^2+o(1).\end{equation}

\bigskip
To establish \eqref{preqforpr}, and thus complete the proof of \eqref{bndonmaxR}, it remains to show that the inequality in \eqref{feqlbndlog} holds in the other direction as well.  To do this, we begin by noting that it follows from Lemma \ref{lemma:compSRWm} that for any $\delta$ with $0<\delta<1$, we have $${\sf SRW}_T\Big(\underset{j\leq 2n}{\text{max}}|X_j|\leq n^{\gamma},\ X_{2n}={\bf 0}\Big)\leq M^n\bigg[c_2^{n/(\text{log}\ n)^{2-\delta}}+{\sf BRW}_T\Big(\underset{j\leq 2n}{\text{max}}|X_j|\leq n^{\gamma},\ B_n\leq\frac{n}{(\text{log}\ n)^{2-\delta}}\Big)\bigg].$$Using this inequality, we see that to complete the proof, it will suffice to show that for some such $\delta$\begin{equation}\label{lastinned}\underset{n\to\infty}{\text{limsup}}\frac{(\text{log}\ n)^2}{n}\text{log}\bigg[{\sf BRW}_T\Big(\underset{j\leq 2n}{\text{max}}|X_j|\leq n^{\gamma},\ B_n\leq\frac{n}{(\text{log}\ n)^{2-\delta}}\Big)\bigg]\leq-\Big(\frac{\pi\text{log}\ m}{\gamma}\Big)^2.\end{equation}In order to establish \eqref{lastinned}, our first step is to define the sequence of times $\{\tau_j\}$.  We do this by letting $\tau_0=0$, and for every $j\geq 1$, defining $\tau_j=\text{min}\{i>\tau_{j-1}:X_i={\bf 0}\ \text{or deg}(X_i)>m\}$.  Letting $N=\text{min}\{j:\tau_j\geq 2n\}$, and then noting that the event $\{B_n\leq\frac{n}{(\text{log}\ n)^{2-\delta}}\}$ can be expressed as $\{N\leq\frac{n}{(\text{log}\ n)^{2-\delta}}\}$, we find that \begin{equation}\label{continlset}\{\underset{j\leq 2n}{\text{max}}|X_j|\leq n^{\gamma}\}\cap\{B_n\leq\frac{n}{(\text{log}\ n)^{2-\delta}}\}\subseteq\{N\leq\frac{n}{(\text{log}\ n)^{2-\delta}}\}\cap\{|X_{\tau_j}|\leq n^{\gamma}\ \forall\ j<N\}.\end{equation}Now note that the properties of $T$ imply that for any $\epsilon>0$, there is an $N_{\epsilon}$ such that for all $n\geq N_{\epsilon}$ \begin{equation}\label{djubnd}\underset{0\leq j\leq n+1}{\text{max}}D_{j,\infty}(T)\leq (1+\epsilon)\frac{\text{log}\ n}{\text{log}\ m}.\end{equation}Taking an arbitrary $\epsilon>0$ and allowing the value $n^{\gamma}$ (as referenced in \eqref{lastinned} and \eqref{continlset}) to be greater than $N_{\epsilon}$, we see that for any $v\in T$ with $|v|\leq n^{\gamma}$, we have \begin{equation}\label{stdombet}(\tau_{j+1}-\tau_j|X_{\tau_j}=v)\preceq 1+\Psi(n,\epsilon),\end{equation}where the expression on the left is in terms of ${\sf BRW}_T$, $\Psi(n,\epsilon)$ represents a random variable equal to the exit time of a simple random walk on $\mathbb{Z}$ starting at $x=1$ from the interval $\big[1,\lfloor{(1+\epsilon)\frac{\gamma\text{log}\ n}{\text{log}\ m}\rfloor}\big]$, and the symbol $``\preceq"$ indicates that the expression on the left in \eqref{stdombet} is stochastically dominated by the expression on the right.  Next we let $\Psi_1(n,\epsilon),\Psi_2(n,\epsilon),\dots$ be a sequence of independent  copies of $\Psi(n,\epsilon)$, and we define the random variable $N':=\text{min}\{j:\Psi_1(n,\epsilon)+\dots +\Psi_j(n,\epsilon)\geq 2n-j\}$.  Combining \eqref{continlset} and \eqref{stdombet}, we observe that \begin{align}\label{bndpuSRW}{\sf BRW}_T\Big(\underset{j\leq 2n}{\text{max}}|X_j|\leq n^{\gamma},\ B_n\leq\frac{n}{(\text{log}\ n)^{2-\delta}}\Big)&\leq{\bf P}\Big(N'\leq\frac{n}{(\text{log}\ n)^{2-\delta}}\Big) \\ &\leq{\bf P}\Bigg(\sum_{j\leq n/(\text{log}\ n)^{2-\delta}}\Psi_j(n,\epsilon)\geq 2n(1-\epsilon)\Bigg)\nonumber,\end{align}
where the second inequality follows from the fact that $2\epsilon n>\frac{n}{(\text{log}\ n)^{2-\delta}}$ for $n$ sufficiently large.  Now using the fact (shown on pg. 16 of \cite{GP}) that for $\delta,\epsilon$ between $0$ and $1$, we have $$\underset{n\to\infty}{\text{limsup}}\frac{(\text{log}\ n)^2}{n}\text{log}\Bigg[{\bf P}\bigg(\sum_{j\leq n/(\text{log}\ n)^{2-\delta}}\Psi_j(n,\epsilon)\geq 2n(1-\epsilon)\bigg)\Bigg]\leq -\frac{(1-\epsilon)^3|\pi\text{log}\ m|^2}{(1+\epsilon)^2\gamma^2},$$ and taking $\epsilon\to 0$, we see that the inequality in \eqref{lastinned} holds for the tree $T$ (and any $0<\delta<1$), thus completing the proof of \eqref{bndonmaxR}.

To establish \eqref{ubndonmaxd}, we first observe that it follows from Lemma \ref{lemma:compSRWm} that for any infinite rooted tree $T$ we have \begin{equation}\label{bndonmaxpiu32}{\sf SRW}_T\Big(\underset{j\leq 2n}{\text{max}}|X_j|\geq\frac{n}{(\text{log}\ n)^{\beta}},\ X_{2n}={\bf 0}\Big)\leq M^n{\bf E}_{{\sf BRW}_T}[c_2^{B_n}1_{H}]\leq M^n{\sf BRW}_T(H)\end{equation}(where $H$ represents the event $\{\underset{j\leq 2n}{\text{max}}|X_j|\geq\frac{n}{(\text{log}\ n)^{\beta}}\}\cap\{X_{2n}={\bf 0}\}$).  Defining $V_n:=\text{min}\{j:|X_j|\geq \frac{n}{(\text{log}\ n)^{\beta}}\}$, we then see that \begin{align}\label{lngstrgineqBRWth}{\sf BRW}_T(H)&=\sum_{j=\lceil{n/(\text{log}\ n)^{\beta}\rceil}}^{2n-\lceil{n/(\text{log}\ n)^{\beta}\rceil}}{\sf BRW}_T(V_n=j)\cdot {\sf BRW}_T(X_{2n}={\bf 0}\ |\ V_n=j)\\&\leq\sum_{j=\lceil{n/(\text{log}\ n)^{\beta}\rceil}}^{2n-\lceil{n/(\text{log}\ n)^{\beta}\rceil}}{\sf BRW}_T(V_n=j)\cdot{\sf SRW}_{\mathbb{Z}}\Big(\text{min}\{i:X_i=\lceil{n/(\text{log}\ n)^{\beta}\rceil}\}\leq 2n-j\Big)\nonumber\\&\leq\sum_{j=\lceil{n/(\text{log}\ n)^{\beta}\rceil}}^{2n-\lceil{n/(\text{log}\ n)^{\beta}\rceil}}{\sf BRW}_T(V_n=j)\cdot{\sf SRW}_{\mathbb{Z}}\Big(\text{min}\{i:X_i=\lceil{n/(\text{log}\ n)^{\beta}\rceil}\}\leq 2n\Big)\nonumber\\&\leq{\sf SRW}_{\mathbb{Z}}\Big(\text{min}\{i:X_i=\lceil{n/(\text{log}\ n)^{\beta}\rceil}\}\leq 2n\Big)\leq 2\cdot{\sf SRW}_{\mathbb{Z}}\Big(X_{2n}\geq\lceil{n/(\text{log}\ n)^{\beta}\rceil}\Big).\nonumber\end{align}Noting, as we did in the proof of Theorem \ref{brbnd0c}, that for $\lambda>0$ sufficiently small we have ${\bf E}_{{\sf SRW}_{\mathbb{Z}}}[e^{\lambda X_j}]\leq e^{\lambda^2 j}$, it then follows from setting $\lambda=\frac{1}{\text{log}\ n}$ and applying Markov's inequality, that ${\sf SRW}_{\mathbb{Z}}\Big(X_{2n}\geq\lceil{n/(\text{log}\ n)^{\beta}\rceil}\Big)\leq e^{2n/(\text{log}\ n)^2-2n/(\text{log}\ n)^{1+\beta}}$, which combined with \eqref{bndonmaxpiu32} and \eqref{lngstrgineqBRWth}, implies that \begin{equation}\label{ubndSRWtmx2n0}{\sf SRW}_T\Big(\underset{j\leq 2n}{\text{max}}|X_j|\geq\frac{n}{(\text{log}\ n)^{\beta}},\ X_{2n}={\bf 0}\Big)\leq 2M^n e^{2n/(\text{log}\ n)^2-2n/(\text{log}\ n)^{1+\beta}}.\end{equation}Now noting that if we take $\gamma\uparrow 1$ in \eqref{preqforpr} then we find that $${\sf SRW}_{{\bf T}}(X_{2n}={\bf 0})\geq M^n e^{-\big(1+o(1)\big)(\pi\text{log}\ m)^2 n/(\text{log}\ n)^2}\ \ {\sf GW}-\text{a.s.},$$it follows from combining this with \eqref{ubndSRWtmx2n0} that $$\underset{n\to\infty}{\text{lim}}{\sf SRW}_{{\bf T}}\Big(\underset{j\leq 2n}{\text{max}}|X_j|\geq\frac{n}{(\text{log}\ n)^{\beta}}\ \Big|\ X_{2n}={\bf 0}\Big)=\underset{n\to\infty}{\text{lim}}\frac{{\sf SRW}_{{\bf T}}\Big(\underset{j\leq 2n}{\text{max}}|X_j|\geq\frac{n}{(\text{log}\ n)^{\beta}},\ X_{2n}={\bf 0}\Big)}{{\sf SRW}_{{\bf T}}\Big(X_{2n}={\bf 0}\Big)}=0\ \ {\sf GW}-\text{a.s.},$$thus establishing \eqref{ubndonmaxd} and completing the proof of the theorem.
\end{proof}

\section*{Acknowledgements}

The author would like to thank Asaf Nachmias for helpful conversations and Marcus Michelen for providing a number of useful suggestions.

\end{document}